\documentclass[a4paper,reqno]{amsart}

\usepackage[centertags]{amsmath}
\usepackage[margin=1in]{geometry}
\usepackage{amsmath}
\usepackage{amsfonts}
\usepackage{amssymb}
\usepackage{verbatim}
\usepackage{enumerate}

\usepackage{graphicx}

\newcommand{\F}{\mathcal{F}}

\newcommand{\I}{\mathcal{I}}

\newcommand{\R}{\mathbb R}
\newcommand{\Z}{\mathbb Z}

\newtheorem{theorem}{Theorem}
\newtheorem{lemma}{Lemma}
\newtheorem{rem}{Remark}

\newtheorem{example}{Example}
\newtheorem{definition}{Definition}

\newtheorem{corollary}{Corollary}

\title[Ising model with flipping sets and fast
decreasing temperature]{Stochastic Ising model with flipping sets of
spins and fast decreasing temperature}

\author{Roy Cerqueti}
\address{University of Macerata, Department of Economics and Law.
Via Crescimbeni 20, I-62100, Macerata, Italy}
\email{roy.cerqueti@unimc.it}

\author{Emilio De Santis}
\address{University of Rome La Sapienza, Department of Mathematics.
Piazzale Aldo Moro, 5, I-00185, Rome, Italy}
\email{desantis@mat.uniroma1.it}

\begin{document}

\begin{abstract}
This paper deals with the stochastic Ising model with a temperature
shrinking to zero as time goes to infinity. A generalization of the
Glauber dynamics is considered, on the basis of the existence of
simultaneous flips of some spins. Such dynamics act on a wide class
of graphs which are periodic and embedded in $\R^d$. The
interactions between couples of spins are assumed to be quenched
i.i.d. random variables following a Bernoulli distribution with
support $\{-1,+1\}$. The specific problem here analyzed concerns the
assessment of how often (finitely or infinitely many times, almost
surely) a given spin flips. Adopting the classification proposed in
\cite{GNS}, we present conditions in order to have models of type
$\mathcal{F}$ (any spin flips finitely many times), $\mathcal{I}$
(any spin flips infinitely many times) and $\mathcal{M}$ (a mixed
case). Several examples are provided in all dimensions and for
different cases of graphs. The most part of the obtained results
holds true for the case of zero-temperature and some of them for the
cubic lattice $\mathbb{L}_d=(\mathbb{Z}^d, \mathbb{E}_d)$ as well.
\end{abstract}

\maketitle

\section{Introduction}

%
%
%
%
%
%
%
%
%
%
%
%


In this paper we deal with a class of non homogeneous Markov
processes $(\sigma(t): t \geq 0)$ in the frame of random
environment. In particular, we consider a generalization of Glauber
dynamics of the Ising model, see e.g. \cite{Liggett}, in the case of
flipping sets whose cardinality is smaller than or equal to a given
$k\in \mathbb{N}$.  However, accordingly to Glauber, also our
dynamics will satisfy the reversibility property with respect to the
Gibbs measure of the Ising model. The Markov process describes the
stochastic evolution of spins, which are binary variables $\pm 1$ on
the vertices of an infinite periodic graph $G=(V,E)$ with finite
degree. Such a class of graphs includes the meaningful standard case
of $d$-dimensional cubic lattices $\mathbb{L}_d=(\mathbb{Z}^d,
\mathbb{E}_d)$, for $d \in \mathbb{N}$. The \emph{interactions}
$\mathcal{J} = (J_e : e \in E )$ are deterministic or i.i.d. random
variables having Bernoulli distribution $ \mu_{\mathcal{J}}( J_e =
+1 ) = \alpha $ and $ \mu_{\mathcal{J}}( J_e = -1 ) =1- \alpha$,
with $\alpha \in [0,1]$, and they constitute the random environment.
The case $\alpha =1$ (resp. $\alpha =0$) corresponds to the
interactions of the homogeneous ferromagnetic (resp.
antiferromagnetic) Ising model.

The \emph{temperature profile} of the model is a function $T:
[0,\infty) \to [0,\infty) $, which is measurable with respect to the
Borel $\sigma$-algebra $\mathcal{B} (\mathbb{R} )$ with $T(t)$
denoting the \emph{temperature at time $t$}, for any $t \in
[0,\infty) $. In most of the results it is assumed that $\lim_{t \to
\infty}T(t)=0$ with a specific reference to $T$ fast decreasing to
zero,  this will be important due to its connections with the well
studied case of $T\equiv 0 $. 


The initial configuration of the Markov process is $\sigma(0) \in
\{-1,+1\}^V$. Its components are assumed to be given or i.i.d. and
randomly selected by a Bernoulli measure $\nu_{\sigma (0)}$ with
parameter $\gamma \in [0,1]$.

In the case of zero temperature stochastic Ising models with
homogeneous interactions, the following question is of particular
relevance:
\begin{itemize}
    \item[\textbf{(Q)}] \emph{Does a given spin on $v \in V $ flip
infinitely many times almost surely?}
\end{itemize}

This problem is indeed paradigmatic in this context (see e.g.
\cite{CDN02,CM06,
DEKNS,DKNS,DMCarlo,DN03,EN,FSS,GNS,HNexagon,Morris,  tessler}).

However, for constant and positive temperature $ T $, question
\textbf{(Q)} does not make sense because all the spins flip
infinitely many times. Moreover, question \textbf{(Q)} should be
rephrased also in the frame of random interactions. In fact, it is
not always meaningful to deal with single sites, because the random
environment leads to sites differently behaving.

In the deep contribution of Gandolfi, Newman and Stein \cite{GNS}
the authors propose also a classification for these models, which is
a partition of them. Specifically, a model is of type $\mathcal{I}$,
$\mathcal{F}$, and $\mathcal{M}$,
 according to
if all the sites flip infinitely many times (a.s.), finitely many
times (a.s.), or some sites flip infinitely many times  and the
others do it finitely many times (a.s.), respectively. We adopt here
the same classification. Question \textbf{(Q)} becomes:

\begin{itemize}
    \item[\textbf{(Q')}] \emph{Is a model of type $\mathcal{I}$,
$\mathcal{F}$ or $\mathcal{M}$?}
\end{itemize}

 We notice that $\mathcal{F}$
corresponds to the almost surely convergence with respect to the
product topology. Vice versa $\mathcal{I}$ and $\mathcal{M}$
correspond to a.s. no-convergence. In this paper we show that some
universal classes can be identified, in the sense that the graph $G
$ and the parameter $k$ determine the class of the model under mild
conditions on $\alpha, \gamma,  T  $. Indeed, the main part of our
results holds true for $\alpha \in (0,1) $ and $\gamma \in [0,1]$
under some natural requirements on the decay rate of the temperature
profile $T$.


\smallskip

In a very different context, \cite{Arratia1} has shown the
recurrence of annihilating random walks on $\Z^d$, with $d \in
\mathbb{N}$, under very general conditions. This paper has as a
consequence that the one-dimensional stochastic Ising model with
$\alpha=1$, $\gamma \in (0,1)$ and $T \equiv 0$ is of type
$\mathcal{I}$ (see \cite{NNS00}).

In \cite{GNS} there is an analysis of the zero-temperature case for
the cubic lattice $\mathbb{L}_d = (\mathbb{Z}^d, \mathbb{E}_d)$,
$\gamma=1/2$ and different product measures $\mu_{\mathcal{J}}$ over
$\R$ for the interactions $ \mathcal{J}  $. Among the results, the
authors provide a complete characterization for $\mathbb{L}_1$ and
for all the measures $\mu_{\mathcal{J}}$, i.e. they identify the
classes $\mathcal{I}$, $\mathcal{F}$, and $\mathcal{M}$ for each
$\mu_{\mathcal{J}}$. In doing so, \cite{GNS} adapts and extends
\cite{Arratia1} in this context. Moreover, the authors identify the
class $\mathcal{M}$ for $\mathbb{L}_2$ when the measure
$\mu_{\mathcal{J}}$ is a product of Bernoulli with parameter $\alpha
\in (0,1)$.

In \cite{NNS00} there is a treatment of $\alpha=1$ and $\gamma=1/2$
for $\mathbb{L}_2$, where it is proven that the model is of type
$\mathcal{I}$. Under the same conditions, \cite{CDN02} refines
\cite{NNS00} in discussing the recurrence and the growth of some
geometrical structures.

It is also worth noting that \cite{NNS00} analyzes the framework
where $\mu_{\mathcal{J}}$  is a continuous measure over $\R$ with
finite mean, and they find $\mathcal{F}$. The finite mean
restriction has been eliminated in the same setting by \cite{DN03},
and the identified class remains $\mathcal{F}$.

The contribution of \cite{FSS} is for $T \equiv 0$, $\alpha=1$,
$\mathbb{L}_d$ with $d \geq 2$ and $\gamma > \gamma^\star_d $, with
$\gamma^\star_d\in (0,1) $. The authors show that, when $\gamma>
\gamma^\star_d $, the value of any spin converges to  $+1$ a.s.,
hence leading to a model of type $\mathcal{F}$. The paper
\cite{Morris} extends \cite{FSS} and shows that $\lim_{d \to \infty}
\gamma^\star_d=1/2$.

In this last context, in order to prove that, for a large value of
$\gamma$, all the spins converge to $+1$, it is worth mentioning
\cite{CM06} where the stochastic Ising model at zero-temperature on
a $d$-ary regular tree $\mathbb{T}_d$ is analyzed. Analogously to
\cite{FSS, Morris}, the authors prove that there exists $\hat
\gamma_d \in (0,1)$ such that  for $\gamma
> \hat \gamma_d $ all the
spins converge to $+1$. Moreover it is also shown, along with other
results, that $\lim_{d \to \infty} \hat \gamma_d =1/2$.

There are important results for the convergence of the system to the
Gibbs state in the case of low temperature (see e.g.
\cite{DS02Glauber,DS03Torpid,DSM,MT}) or high temperature (see e.g.
\cite{DS02Glauber,DSL,MOS,Mathieu}). In particular, in
\cite{DS02Glauber}, in the framework of random interactions, it is
shown a different approach to the equilibrium measure in relation to
the temperature. We believe that there is a connection between the
behaviour of the stochastic Ising model with $\alpha \in (0,1)$ at
low, constant and positive temperature and at zero-temperature. In
this respect, it seems that the type of the model, $\F$ or
$\mathcal{M}$, at zero-temperature is related to properties of
metastability of the same model at low constant temperature. The
geometric structure of the graph is also relevant in both the cases
of zero- and low constant temperature. In \cite{DS02Glauber,GNS}
some of the main results come out from a combinatorial-geometric
interpretation of the underlying graphs.

For cases different from $\mathbb{L}_d$ at zero-temperature, we
mention \cite{CM06} and \cite{EN}, with tree-related graphs;
\cite{tessler}, where trees and cylinders originating by graphs are
considered; \cite{HNexagon}, where the hexagonal lattice is
explored. In this latter
 paper the authors move from \cite{NNS00}, where it is proven that
sites fixate, and show that the expected value of the cardinality of
the cluster containing the origin becomes infinite when time grows.

As already announced above, we aim to provide an answer to
\textbf{(Q')} in the framework of  general graphs. In particular, we
deal with graphs which are periodic and embedded in $\R^d$.
In so doing, we contribute to the
 literature, which is mainly focussed on $\mathbb{L}_1$ (see e.g.
\cite{GNS}), $\mathbb{L}_2$ (see e.g. \cite{CDN02,GNS, NNS00}) and
the hexagonal lattice (see e.g. \cite{HNexagon}) with the 
exceptions of the general dimensional cubic lattices $\mathbb{L}_d$
in \cite{FSS,Morris}.

Notice that this topic is important either at a purely theoretical
level as well as in the applied science. Indeed, we mention
\cite{DMCarlo,Mendes,Stauffer,Tamuz}, where applications of Ising
models to social science are presented. In particular,
\cite{DMCarlo} deals with Glauber dynamics at zero-temperature over
random graphs, where nodes are social entities and the spatial
structure captures the social connections among the nodes. For a
review of the relevant contribution on the so-called
\emph{sociophysics}, refer to \cite{Stauffer}.

This paper adds to the literature on the stochastic Ising models as
follows:
\begin{itemize}
\item[$(i)$] We allow for simultaneous flips of the spins, hence
allowing for flipping regions. This statement has an interest under
a theoretical perspective and it seems to be also reasonable for the
development of real-world decision processes (think at the changing
of opinions process of groups of connected agents rather than of
single individuals). The cardinality of the flipping sets is
constrained by a parameter $k$, which will be formalized below. Some
material related to this aspect can be found in \cite{tessler}.
\item[$(ii)$] The considered graphs are periodic, infinite and with
finite degree. This framework naturally includes  $\mathbb{L}_d$,
for each $d \in \mathbb{N}$. This generalization allows us to
provide theoretical results and physically consistent examples (like
crystal lattices) outside the restrictive world of $\mathbb{L}_d$.
\item[$(iii)$] The temperature is taken not necessarily zero.
Specifically, we consider a temperature fast decreasing to zero and
 we require only in one result its positivity. In doing so, we develop
a theory on the reasonable situation of a temperature changing
continuously in time, without assuming the jump from infinite to
zero. The framework of temperature fast decreasing to zero includes
also the case of $T \equiv 0$.
\end{itemize}
In more details, Lemma \ref{serve} is a technical result giving the
framework we deal with. Specifically, it provides some grounding
consequences of the definition of temperature profile fast
decreasing to zero, which are useful in the rest of the paper.

Theorem \ref{opposition} gives sufficient conditions for the
temperature profile having a similar or different behaviour of the
zero-temperature case. It is shown that the model is of type
$\mathcal{I}$ under an asymptotic condition for the temperature
profile. This result depends on the Hamiltonian, and can then be
rewritten for general Gibbs models endowed with a Glauber-type
dynamics. Some conditions of Theorem \ref{opposition} can be
replaced with weaker ones  over graphs with even-degree sites (e.g.
$\mathbb{L}_d$).

In Theorem \ref{infiniti} we present a result stating that the model
is of type $\mathcal{M}$ or $\I$, under some hypotheses.
Specifically, it is required that the temperature profile $T$ is
fast decreasing to zero, positive and the graph belongs to the
rather wide class of $k$-\emph{stable}
$d$-$\mathcal{E}$\emph{graphs} (see next section for the formal
definition of this concept). Such a class contains also the cubic
lattices $\mathbb{L}_d$, with $d \geq 2$. We believe that Theorem
\ref{infiniti} represents our main result because it is a relevant
step to the identification of the type of stochastic Ising models in
dimension $d\geq 2$. We stress that the positivity of the
temperature is required only in this result.

Theorem \ref{unione} states some conditions to obtain models of type
$\mathcal{M}$, and it is used to prove some findings in Section
\ref{rhoFno0}.

Theorem \ref{hexag} shows that the model with ferromagnetic
interactions and $\gamma= 1/2$ on the  hexagonal lattice is not of
type $\mathcal{F}$ when $k \geq 2$. In so doing we complement
\cite{NNS00}, where it is proven that the same model with $k=1$ and
$T \equiv 0$ is of type $\mathcal{F}$.

In Section \ref{rhoFno0} we describe conditions leading to fixating
sites. Lemmas \ref{Lemrichiesto} and \ref{Lemnonrichiesto} highlight
the connection between lowering and increasing energy flips. Indeed,
under the hypothesis of Lemma \ref{serve}, the number of such flips
is a random variable having finite mean on any given set of flipping
spins. These two Lemmas are widely used to prove the results of this
Section.

Theorem \ref{culta}  is inspired by \cite{GNS,NNS00} and generalizes
their outcomes to the case of $k>1$. In particular, it links the
parameter $k$ with the properties of the graph to obtain that some
sites fixate, i.e. the model is of type $\mathcal{M}$ or
$\mathcal{F}$. Theorem \ref{cultb} formalizes the intuitive fact
that if there exists a set of sites strongly interconnected and
weakly connected with the complement of the set, then these sites
will fixate with positive probability. 

In Definition \ref{defeabsent} we adapt to our setting the concept
of \emph{e-absent} configurations (namely, $k$-\emph{absent on}
$\mathcal{J}$ here), which has been introduced for the graph
$\mathbb{L}_2$ in \cite{CDN02,GNS}. We generalize such a definition
by including $k
>1$ and the considered class of graphs. Theorems \ref{teoabsent} and
\ref{teorhoF} are based on $k$-absence. The former result states
that configurations which are $k$-absent on $\mathcal{J}$  can
appear only on a random finite time interval (a.s.) when
interactions are properly selected; the latter one provides a
condition on the graph and the parameter $k$ such that some sites
have a positive probability to fixate. An interesting consequence of
the definition of $k$-absence is that if a configuration is
$k$-absent on $\mathcal{J}$, then it is $k'$-absent on
$\mathcal{J}$, for each $k'>k$. Hence, a large value of $k$ seems to
facilitate the fixation of the sites (see  Theorem \ref{teoabsent}).
Differently, a large value of $k$ is an obstacle for the fixation of
the sites in Theorem \ref{teorhoF}. This contrast leads to a not
straightforward link between the value of $k$ and the identification
of the model. However, in our setting, we have shown that a model on
the hexagonal lattice is of type $\mathcal{F}$ for $k=1$ (see also
\cite{NNS00}) and it is not of type $\mathcal{F}$ for $k\geq 2$ (see
Theorem \ref{hexag}). This Theorem suggests a general conjecture
that link the value of $k$ with the type of the model (see the
Conclusions).

The provided examples illustrate a wide part of the outcomes, and
complement the theoretical findings of the paper. In particular, we
have introduced the graph $\Gamma_{\ell,m}(G)$, it is constructed by
replacing the original edges of a graph $G=(V,E) $ with more complex
structures (see Definition \ref{def:Gamma}). For this specific
class, we have provided some conditions for which the models over
such graphs are of type $\mathcal{M}$ (see Theorem \ref{thm:ultimo})
and $\mathcal{I}$ (see Theorem \ref{thm:ultimissimo}). The case
$\mathcal{F}$ is left as a conjecture in the Conclusions.

The last section of the paper concludes and offers some conjectures
and open problems. In order to assist the reader, we have provided
some figures presenting
 the main contributions of the
related literature at zero-temperature, the results obtained in the
present paper for the $\Gamma_{\ell,m}(G)$ graphs and for $k$-stable
$d$-$\mathcal{E}$ graphs (see Figure \ref{fig5}).

\section{Definitions and first properties of the model} \label{sec2}

The main target of this section is to define a dynamical stochastic
Ising model. In order to do it we introduce some notation.

\subsection{Graphs}



A graph $G=(V,E)$ with origin $O$ is said to be a \emph{$d$-graph}
if
\begin{itemize}
\item[1.]   $G$ is embedded in $\mathbb{R}^d$,
i.e. the vertices are points and the edges are line segments;
\item[2.]  $G$ is translation invariant with respect
to the vectors $e_i$ of the canonical basis of $\mathbb{R}^d$;
    \item[3.]  for all  $v\in V$, the degree of $v$, denoted by $d_v$, is such that $d_v <
    \infty$;
    \item[4.]  every finite region $ S \subset \mathbb{R}^d $ contains
    finitely many
     vertices of  $G$.
\end{itemize}

Thus we can construct a $d$-graph $G =(V,E)$ doing a
tessellation of $\mathbb{R}^d $ with the basic cell $Cell =[0, 1 )^d
$, i.e.  the space  $\mathbb{R}^d $ is seen as the union of disjoint
hypercubes $ [0, 1 )^d + z   $ with $z \in \mathbb{Z}^d $. 
In order to
specify the $d$-graph $G = (V,E)$ it is enough to give the vertices
$V_{Cell}$ inside $Cell $ with the edges
$$
E_{Cell} = \{\{x,y \}\in E: \{x,y\} \cap Cell \neq \emptyset \}.
$$

Set $ d_G = \sup  \{ d_v : v \in V \} $ \emph{the maximal degree} of
the $d$-graph $G =(V, E)$; by periodicity and the fact that $Cell$
contains finitely many  vertices, then $d_G$ is finite.

If a  $d$-graph has at least a vertex with even degree
we denote it as  $d$-$\mathcal{ E}$ \emph{graph}. As an example of a
$d$-$\mathcal{ E}$ graph we take the lattice $\mathbb{L}_d
=(\mathbb{Z}^d, \mathbb{E}_d)$, 
where the degree of any vertex is equal to $ 2d$.

\medskip

Some standard definitions on graph theory are now recalled. Given $n
\in \mathbb{N}$, a \emph{path  of length $n$} starting in $x \in V $
and ending in $y \in V $ is a sequence of vertices  $ (x_0 =x , x_1,
\ldots , x_{n -1}, x_n =y ) $ having $\{ x_{i-1}, x_{i}\} \in E$ for
each $i = 1, \ldots , n$. A set $A \subset V $ is \emph{connected}
if for any couple $x,y \in A $ there exists a path $ (x_0 =x , x_1,
\ldots , x_{n -1}, x_n =y ) $ with $ x_i \in A $, for $i = 0 ,
\ldots , n $. We say that  a graph $G =(V, E ) $ is connected if $V
$ is connected.

Since, without loss of generality,  one can study separately the
different connected components, as will be clear below, in the
sequel we will consider only connected
$d$-graphs. 


The distance $\nu_G (u,v )  $ in $G$ of two vertices $u, v \in V $
is the length of the shortest path (not necessarily unique) starting
in $u $ and ending in $v$. For $u \in V $ and $L \in \mathbb{N}$ we
define the ball centered in $u $ with radius $L $ as
\begin{equation*}\label{palla}
    B_L(u )= \{ v \in V :\nu_G (u,v ) \leq L \} .
\end{equation*}
The \emph{boundary} of a set $A \subset V $ is
\begin{equation*}\label{frontierainter}
    \partial A =\{ v \in A: \exists u \not \in A \text{ such that } \{u, v\} \in
E\},
\end{equation*}
whereas the \emph{external boundary} is
$$
\partial^{ext} A =\{ v \notin A: \exists u  \in A \text{ such that } \{u, v\} \in
E\}.
$$

\subsection{Hamiltonian}

Let us consider the space $ \{-1,+1\}^V$ equipped with the product
topology. For a given \emph{configuration} $\sigma \in \{-1,+1\}^V$
and for any subset $A$ of $V$, we write $\sigma^{(A)}
=(\sigma^{(A)}_v: v \in V )$ to denote the configuration
\begin{equation*} \label{soloinA}
\sigma^{(A)}_y = \left \{ \begin{array}{cc}
                          -\sigma_y, & \text{if } y \in A;  \\
                          \sigma_y, &  \text{if } y \not \in  A;
                        \end{array} \right .
\end{equation*}
that corresponds to \emph{flip} the configuration $\sigma $ on the
set $A $.

The \emph{formal Hamiltonian} associated to the \emph{interactions}
$\mathcal{J} \in \{-1,+1\}^E$ is
 \begin{equation} \label{hamiltonian}
\mathcal{H}_{\mathcal{J}}(\sigma)  = -\sum_{e=\{x , y\}\in E } J_{e} \sigma_x  \sigma_y.
\end{equation}
However, the definition \eqref{hamiltonian} is not well posed for
infinite graphs. Thus, we shall work with the \emph{increment of the
Hamiltonian at} $A$, namely for a finite set   $A \subset V$
\begin{equation} \label{DeltaHonA}
\Delta_A \mathcal{H}_{\mathcal{J}}(\sigma)  = - 2 \cdot \sum_{e =\{x
, y\}\in E : x\in \partial A, y \in \partial^{ext} A } J_{e }
\sigma^{(A)}_x \sigma^{(A)}_y = 2 \cdot\sum_{e =\{x , y\}\in E: x\in
\partial A, y \in \partial^{ext} A}
 J_{e }\sigma_x \sigma_y.
\end{equation}


Notice that the value of $\Delta_A
\mathcal{H}_{\mathcal{J}}(\sigma)$ can be only an even integer. In
particular, if $\Delta_A \mathcal{H}_{\mathcal{J}}(\sigma) \not =
0$, then $|\Delta_A \mathcal{H}_{\mathcal{J}}(\sigma)|\geq 2$.

\subsection{Dynamics}
 The dynamics of the system will be a non-homogeneous Markov process
depending on the interactions, the temperature profile and the
initial configuration.

The process is denoted by $\sigma(\cdot )= (\sigma_v (t): v \in V ,
t\in[0, \infty))$. It takes value in $\{-1,+1\}^{V} $ and  has left
continuous trajectories.

For $k \in \mathbb{N}$, we call $\mathcal{A}_k $ the collection of
the connected subsets $A \subset V$ having cardinality smaller or
equal to $k $. It is important to stress that, for a given $v \in
V$, the set $\{A \in \mathcal{A}_k : v \in A\}$ is finite for a
$d$-graph. Therefore, the set $\mathcal{A}_k $ is countable and,
following \cite{Liggett}, one  defines the infinitesimal generator
related to $k $  as follows
\begin{equation} \label{generatorA}
\mathcal{L}^{k, \mathcal {J} ,T }_t( f(\sigma) ) = \sum_{A \in
\mathcal{A}_k }
c^{\mathcal{J},T , (A)}_t(\sigma ) (f (\sigma^{(A)})  - f (\sigma)
),
\end{equation}
where  $f: \{ -1,+1\}^V \to \mathbb{R}$ is a continuous function,
$T$ is the temperature profile,  $\mathcal{J} \in \{ -1,+1\}^E$ and
$ \sigma \in \{ -1,+1\}^V$. Moreover, when $T(t)>0$, the rates are
\begin{equation} \label{cA}
c^{\mathcal{J},  T, (A)}_t(\sigma ) = \frac{ e^{  - \Delta_A
\mathcal{H}_{\mathcal{J}} (\sigma ) / T(t)} }{ e^{   \Delta_A
\mathcal{H}_{\mathcal{J}} (\sigma ) / T(t) } +e^{  - \Delta_A
\mathcal{H}_{\mathcal{J}} (\sigma )/ T(t) }  } = \frac{ 1 }{1+  e^{2
\Delta_A \mathcal{H}_{\mathcal{J}} (\sigma ) / T(t) } } .
\end{equation}
In the case in which $T(t ) =0 $ for $t \geq 0$, set
\begin{equation*} \label{cAzero}
c^{\mathcal{J}, T,(A)}_t(\sigma ) = \lim_{T \to 0^+}  \frac{ e^{  -
\Delta_A \mathcal{H}_{\mathcal{J}} (\sigma ) / T} }{ e^{   \Delta_A
\mathcal{H}_{\mathcal{J}} (\sigma ) / T } +e^{  - \Delta_A
\mathcal{H}_{\mathcal{J}} (\sigma )/ T } } .
\end{equation*}

Hence, when $T (t)= 0 $, one has
\begin{equation*}\label{zerotemp}
    c^{\mathcal{J},T,(A)}_t(\sigma ) =\left \{
    \begin{array}{cc}
      0 &\text{ if $\Delta_A \mathcal{H}_{\mathcal{J}} (\sigma )>0 $;}  \\
      \frac{1}{2} &\text{ if $\Delta_A \mathcal{H}_{\mathcal{J}} (\sigma )=0 $;}  \\
      1 &\text{ if $\Delta_A \mathcal{H}_{\mathcal{J}} (\sigma )<0 $.}  \\
    \end{array}
    \right .
\end{equation*}

\medskip

We remark that in the case of $T\equiv 0$ and $k =1$ our dynamics is
the Glauber dynamics at zero temperature (see for instance
\cite{CDN02, DN03, GNS, Morris,NNS00}).

The previous defined process can be constructed by using a
collection of independent Poisson processes $( \mathcal {P}_A : A
\in \mathcal{A}_k)$ with rate $ 1$, the so-called Harris' graphical
representation \cite{Harris}.

Call $\mathcal{T}_A =(\tau_{A, n } : n \in \mathbb{N}  )$  the
arrivals of the Poisson process $\mathcal{P}_A$.  The probability
that there is a \emph{flip at the set $A$} (conditioning on the
event $\{ t \in \mathcal{T}_A \} $), i.e. $ \sigma (t^+) =\lim_{s
\downarrow t} \sigma (s ) =\sigma^{(A)} (t) $, is equal to $
c^{\mathcal{J},T,(A)}_t(\sigma (t)) $, where we pose $ \sigma^{(A)}
(t) = (\sigma(t))^{(A)}  $. An useful representation of such events
can be given through the family of i.i.d. random variables
\begin{equation}
\label{UAn} (U_{A,n}: A \in \mathcal{A}_k, n \in \mathbb{N})
\end{equation}
which are uniform in $[0,1]$ and such that if $U_{A,n} <
c^{\mathcal{J} ,T,(A)}_{\tau_{A,n}} (\sigma (\tau_{A,n}) )$, then
there is a flip at the set $A$ at time $\tau_{A,n}$ (see
\cite{Harris, Liggett}).

The representation of the Markov process based on the Poisson
processes is very popular in the framework of zero-temperature
dynamics, since it exhibits remarkable advantages with respect to
the one based on the generator. Firstly, the spatial ergodicity of
the process with respect to the translations related to the
canonical
basis 
can be invoked on the ground of a very general theory (see
\cite{Harris,ergodic, NNS00, Sepp}). Secondarily, this
representation is the natural setting for the proofs of the results.

\medskip

We say that a flip at $A $ at time $ t \in  \mathcal{T}_A$ is in
\emph{favour of the Hamiltonian}  if $ \Delta_A
\mathcal{H}_{\mathcal{J}}  (\sigma (t) ) <0 $; it is
\emph{indifferent for the Hamiltonian}  if $ \Delta_A
\mathcal{H}_{\mathcal{J}} (\sigma (t) ) =0 $ and it is \emph{in
opposition of the Hamiltonian}  if $ \Delta_A
\mathcal{H}_{\mathcal{J}}  (\sigma (t)) >0 $.

\bigskip

For $k \in \mathbb{N}$, $t \in [0,\infty]$ and $A \in
\mathcal{A}_k$, we define the sets $\mathcal{S}^-_{A,t}$,
$\mathcal{S}^0_{A,t}$ and $\mathcal{S}^+_{A,t}$ as
\begin{equation}
\label{SAt-} \mathcal{S}^-_{A,t}=\left\{s \in [0,t] \cap
\mathcal{T}_A : \text{ there is a flip at $A$ at time $s$ with }
\Delta_A \mathcal{H}_{\mathcal{J}}(\sigma(s))> 0 \right\},
\end{equation}
\begin{equation}
\label{SAt0} \mathcal{S}^0_{A,t}=\left\{s \in [0,t] \cap
\mathcal{T}_A : \text{ there is a flip at $A$ at time $s$ with }
\Delta_A \mathcal{H}_{\mathcal{J}}(\sigma(s))= 0 \right\},
\end{equation}
\begin{equation}
\label{SAt+} \mathcal{S}^+_{A,t}=\left\{s \in [0,t] \cap
\mathcal{T}_A :\text{ there is a flip at $A$ at time $s$ with }
\Delta_A \mathcal{H}_{\mathcal{J}}(\sigma(s))< 0 \right\},
\end{equation}
and, for any $x \in V$,
\begin{equation}\label{N-}
\mathcal{N}^-_x =\bigcup_{A \in \mathcal{A}_k: A \ni x }
\mathcal{S}^-_{A,\infty} , \text{ }
 \mathcal{N}^0_x =\bigcup_{A \in
\mathcal{A}_k: A \ni x } \mathcal{S}^0_{A,\infty} , \text{ }
\mathcal{N}^+_x =\bigcup_{A \in \mathcal{A}_k: A \ni x }
\mathcal{S}^+_{A,\infty}.
\end{equation}
Moreover,
we define \emph{the set of total flips involving the site $x$} as
$$
\mathcal{N}_x =\mathcal{N}^-_x \cup  \mathcal{N}^0_x \cup
\mathcal{N}^+_x ,
$$
and \emph{the set of total arrivals  involving the site $x$} as
$$
\mathcal{Q}_x  =\bigcup_{A \in \mathcal{A}_k: A \ni x } \{ t \in \mathcal{T}_A \}.
$$

\medskip

We now provide the definition of the probability measure associated
to the dynamical model defined in \eqref{generatorA}, which can be
written as
\begin{equation*}
\label{P} P_{k, \mu_{\mathcal{J}} , \nu_{\sigma(0)}}
=\mu_{\mathcal{J}} \times \nu_{\sigma(0)} \times P_{k},
\end{equation*}
where $\mu_{\mathcal{J}}$ is the probability distribution of the
interactions $\mathcal{J}$ forming the quenched random environment;
$\nu_{\sigma(0)}$ is the probability distribution of the initial
configuration $\sigma(0)$; $P_{k}$ is the measure of the arrivals of
the i.i.d. Poisson processes on the sets $A \in \mathcal{A}_k$ and
the sequences of independent $U$'s as in \eqref{UAn} -- independent
also from the Poisson processes.

Some particular cases are important in our context. The considered
measures over the interactions are of two families
\begin{equation}
\label{mu=Bernoulli}\mu_{\mathcal{J}}=\prod_{e \in E} Ber_e(\alpha),
\qquad \alpha \in [0,1],
\end{equation}
where $Ber_e(\alpha)$ is the Bernoulli distribution with parameter
$\alpha$ and support $\{-1,+1\}$, labeled by the edge $e\in E$. The
deterministic case is
\begin{equation}
\label{mu=Jfissato}\mu_{\mathcal{J}}=\delta_{\mathcal{J}}, \qquad
\mathcal{J} \in \{-1,+1\}^E.
\end{equation}
Measure $\mu_{\mathcal{J}}$ in \eqref{mu=Bernoulli} is the case of
$\pm J$ model in \cite{GNS}, while \eqref{mu=Jfissato} is the case
of deterministically selecting the interactions on the edges of the
graph.

Analogously, for the initial configuration $\sigma(0)$, we use
\begin{equation}
\label{nu=Bernoulli}\nu_{\sigma(0)}=\prod_{v \in V} Ber_v(\gamma),
\qquad \gamma \in [0,1];
\end{equation}
\begin{equation}
\label{nu=fissato}\nu_{\sigma(0)}=\delta_{\sigma}, \qquad \sigma \in
\{-1,+1\}^V.
\end{equation}

To avoid a cumbersome notation, we will denote the probability
measure $P_{k, \mu_{\mathcal{J}} , \nu_{\sigma(0)}}$ simply as
$\mathbb{P}$, and the identification of it will be clear from the
context. Analogously, the expected value associated to the
probability measure $P_{k, \mu_{\mathcal{J}} , \nu_{\sigma(0)}}$
will be indicated as $\mathbb{E}$.

Let us consider a
$d$-graph $G=(V,E)$. Then,
\begin{itemize}
\item when $\mu_{\mathcal{J}}$ and
$\nu_{\sigma(0)}$ are as in \eqref{mu=Bernoulli} and
\eqref{nu=Bernoulli}, respectively, then the process will be denoted
as $(k, \alpha , \gamma; T)$-model on $G$, for a temperature profile
$T$;
\item when $\mu_{\mathcal{J}}$ and
$\nu_{\sigma(0)}$ are as in \eqref{mu=Jfissato} and
\eqref{nu=fissato}, respectively, then the process will be denoted
as $(k, \langle \mathcal{J} \rangle, \langle\sigma\rangle; T)$-model
 on $G$, for a temperature profile $T$.
\end{itemize}
\begin{rem}
To give the interactions $\mathcal{J}$ and the initial configuration
$\sigma(0)$ (hence, leading to the indices
$\langle\mathcal{J}\rangle$ and $\langle\sigma\rangle$,
respectively) has an important role in the proof of some results, as
we will see below. In fact, in the considered framework, when a
property is shown for all $\mathcal{J}$ and $\sigma(0)$, then such a
property holds true in the Bernoullian case as well.
\end{rem}

\medskip
The definition of the measure associated to the model allows us to
introduce the quantities of interest in assessing its type.

By adopting the notation of \cite{GNS}, a model is said to be of
type $\mathcal{I}$, $\mathcal{F}$, and $\mathcal{M}$, if all the
sites flip infinitely many times (a.s.), finitely many times (a.s.),
or some sites flip infinitely many times and the others do it
finitely many times (a.s.), respectively.

In particular, using standard ergodic arguments one can see that for
a $(k, \alpha , \gamma; T)$- model on a $d$-graph $G$ the quantity
\begin{equation*}\label{rhoI}
\rho_{\mathcal{I}} = \rho_{\mathcal{I}} (k, \alpha , \gamma;T )=
\lim_{\ell \to \infty } \frac{| \{x \in  B_\ell (v) : |\mathcal{N}_x| =
\infty  \} |}{ | B_\ell (v) | }
\end{equation*}
does exist and it is constant almost surely, and it does not depend
on the vertex $v$
(for details 
see \cite{GNS,NNS00}).

We also define
$$
\rho_{\mathcal{F}} =\rho_{\mathcal{F}} (k, \alpha , \gamma ; T)=
 \lim_{\ell \to \infty } \frac{| \{x \in  B_\ell(v) :  |\mathcal{N}_x| <
\infty  \} |}{ | B_\ell (v) | }.
$$
Therefore $\rho_{\mathcal{I}}+\rho_{\mathcal{F}} =1 $, and
\begin{itemize}
\item if $\rho_{\mathcal{I}}=1$ ($\rho_{\mathcal{F}}=1$, resp.) then the
$(k, \alpha , \gamma ; T)$-model on the $d$-graph $G$ is of type
$\mathcal{I}$ (${\mathcal{F}}$, resp.);
\item if $0<\rho_{\mathcal{I}}<1$ then the
$(k, \alpha , \gamma ; T)$-model on the $d$-graph $G$
 is of type $\mathcal{M}$.
\end{itemize}

\section{Conditions for $\rho_{\mathcal{I}} >0$}

The first results are given in a general setting, where the
interactions $\mathcal{J} $ and the initial configuration $\sigma
(0)$ are provided; moreover, they could be presented  for a
completely general Glauber dynamics associated to a Gibbs measure
and, thus, associated to a Hamiltonian.

First, we give the following uniform integrability condition
\begin{definition}    \label{def:T} Let us consider a $d$-graph $G=(V,E)$.
%
%
The temperature profile $T $ is said to be \emph{fast decreasing to
zero} if, for a $ (k, \langle \mathcal{J} \rangle,
\langle\sigma\rangle; T)$-model on
 $G$, one has
$$
\lim_{t \to \infty} \sup_{\hat \sigma \in \{-1,+1\}^V} \mathbb{E}  (
|\mathcal{N}_x^- \cap [t, \infty )| \,\,|\,\, \sigma (t) = \hat
\sigma ) =0,
$$
for any $x \in V$, $k \in \mathbb{N}$, $\mathcal{J} \in
\{-1,+1\}^E$. 

\end{definition}

Notice that the condition of being fast decreasing to zero is an
asymptotic property of the temperature profile, and the complete
knowledge of the behavior of $T$ is not required.

An immediate consequence of the previous definition is the following
\begin{lemma} \label{serve}
Let us take a $ (k, \langle \mathcal{J} \rangle,
\langle\sigma\rangle; T)$-model on a $d$-graph $G= (V, E)$ with $T $
fast decreasing to zero. Then
\begin{equation} \label{val1}
\mathbb{E} ( |\mathcal{N}_x^- | ) <\infty , \text{ for any } x \in
V.
\end{equation}
Furthermore, for any finite set $V_0\subset V  $,
\begin{equation} \label{eq:thm1new}
\lim\limits_{t \to \infty}\inf_{\hat \sigma \in \{-1,+1\}^V}
\mathbb{P} (  \bigcap_{x \in V_0}\{ \mathcal{N}^-_x \cap [t , \infty
)=\emptyset   \}      \,\, | \,\, \sigma (t) = \hat \sigma )=1,
\end{equation}
and
\begin{equation}\label{fa}
 \mathbb{P}  \left
  ( \bigcap_{x \in V_0}\{  \mathcal{N}^-_x=\emptyset\}
 \right )  >0 .
\end{equation}
\end{lemma}
\begin{proof} \label{coro3}
Let us define
\begin{equation} \label{kappa}
K_x =  |\{ A \in \mathcal{A}_k:  A \ni x\} |, \text{ } K = \max_{x
\in V} K_x   .
\end{equation}

Clearly, for any $x \in V$,  $K_x $ is finite and also $K $ is
finite because it corresponds to take the maximum only inside
$Cell$.

Since $T $ is fast decreasing to zero, we can take $ t_0 $
sufficiently large to have
$$
\sup_{\hat \sigma \in \{-1,+1\}^V} \mathbb{E}( |\mathcal{N}_x^- \cap
[t_0, \infty )| \,\,|\,\, \sigma (t_0) = \hat \sigma ) \leq 1 .
$$

One has
$$
\mathbb{E}    ( |\mathcal{N}_x^- | ) =\mathbb{E}   (
|\mathcal{N}_x^- \cap [ 0,t_0) | ) + \mathbb{E}  ( |\mathcal{N}_x^-
\cap [t_0, \infty ) | )
$$
$$
\leq  \mathbb{E}    ( |\mathcal{Q}_x \cap [ 0,t_0) | ) + \sup_{\hat
\sigma \in \{-1,+1\}^V}\mathbb{E}( |\mathcal{N}_x^- \cap [t_0,
\infty )| \,\,|\,\, \sigma (t_0) = \hat \sigma ) \leq K t_0 + 1,
$$
and \eqref{val1} is proved.

 To prove \eqref{eq:thm1new} we notice
that, by continuity of the measure and being $T$ fast decreasing to
zero, for each $\varepsilon >0$ there exists $t_0
>0 $ such that for all $t \geq t_0$ and $x \in V_0$
$$
\inf_{\hat \sigma \in \{-1,+1\}^V}     \mathbb{P} ( \mathcal{N}^-_x
\cap [t , \infty )=\emptyset \,\, | \,\, \sigma (t) = \hat \sigma )
> 1-\varepsilon,
$$
and thus
\begin{equation*}
\label{eq:nuova} \pi_0 =\inf_{\hat \sigma \in \{-1,+1\}^V}
\mathbb{P} \left
  ( \bigcap_{x \in V_0} \{ \mathcal{N}^-_x \cap
  [t , \infty )  =\emptyset \}  \,\, \Big| \,\, \sigma (t) = \hat \sigma
 \right )  \geq 1-\varepsilon |V_0| .
\end{equation*}
This leads to \eqref{eq:thm1new}.


Finally, by the Markov property
$$
  \mathbb{P}  \left
  ( \bigcap_{x \in V_0} \{ \mathcal{N}^-_x   =\emptyset\}
 \right )  \geq   \mathbb{P}  \left
   ( \bigcap_{x \in V_0} \{ \mathcal{N}^-_x \cap    [0, t_0  ) =\emptyset\}   \right )  \pi_0 .
$$
But
$$
\mathbb{P}  \left
   ( \bigcap_{x \in V_0} \{ \mathcal{N}^-_x \cap    [0, t_0  ) =\emptyset\}   \right )
   \geq
    \mathbb{P} \left
      ( \bigcap_{x \in V_0} \{  \mathcal{Q}_x \cap    [0, t_0  )  =\emptyset\}   \right )
      \geq e^{- K |V_0 | t_0 }  >0 .
$$
For $\varepsilon < \frac{1}{|V_0|}$, one obtains  \eqref{fa}.
\end{proof}

\begin{theorem} \label{opposition}
Let us consider a $ (k, \langle \mathcal{J} \rangle,
\langle\sigma\rangle; T)$-model on a $d$-graph $G= (V, E)$.
\begin{itemize}
  \item  If
\begin{equation}
\label{cond-item1}
 \limsup_{t \to \infty }  T(t)  \ln t < 4 ,
\end{equation}
 then  the temperature profile $T $ is  fast decreasing to zero.
  \item If, for a given $x \in V$,
\begin{equation}
\label{cond-item2} \liminf_{t \to \infty }  T(t)  \ln t >4d_x,
\end{equation}
then $\mathcal{N}_x$ is infinite almost surely.
\end{itemize}
\end{theorem}
\begin{proof} \label{primaprova}
We start by proving the first item of the theorem.

By \eqref{cond-item1} one has
\begin{equation}\label{fraz}
    r = \frac{1}{4} \cdot \limsup_{t \to \infty } T(t)  \ln t < 1.
\end{equation}

We also consider a constant $ a   \in (1,\frac{1}{r } )  $. Given
$\mathcal{N}^-_{x}$ (resp. $\mathcal{N}_{x}$), for $x \in V$ and $n
\in \mathbb{N}$, we define $\mathcal{N}^-_{x,n } = \mathcal{N}^-_{x}
\cap [ n-1, n  )$ (resp. $\mathcal{N}_{x,n } = \mathcal{N}_{x} \cap
[ n-1, n  )$).

Thus, there exists $\bar n \in \mathbb{N}$ large enough such that,
for any $n\geq \bar n $,  $t \in [n-1,n )$, $\mathcal{J} \in
\{-1,+1\}^E$, $\sigma \in \{-1,+1\}^V$ one has
\begin{equation}\label{disu-c}
    \frac{1}{(n-1)^{a}} \geq \sup \left \{
c^{\mathcal{J},T,(A)}_t(\sigma ):c^{\mathcal{J},T,(A)}_t(\sigma ) <
\frac{1}{2}  \right \}, \text{ for } A \ni x  .
\end{equation}

In fact, for any $t \in [n-1,n )$ using \eqref{cA} and the fact that
 $\frac{1}{1+e^x}$ decreases in $x$, one has
\begin{equation}
\label{cAtheo1} c^{\mathcal{J},T,(A)}_t(\sigma )<\frac{1}{2}
\,\,\,\Rightarrow \,\,\,c^{\mathcal{J},T,(A)}_t(\sigma
)<\frac{1}{1+e^{4/T(t)}}.
\end{equation}
By \eqref{fraz} and for a given $r ' \in ( r , 1)$, there exists
$t^\prime$ such that for any $t> t^\prime$ then $ 4/T(t)>\ln t / r
'$. Now, it is sufficient to take $ \bar{n} = \lceil t^\prime \rceil
+1$.

The second inequality in \eqref{cAtheo1} gives that, for any $n\geq
\bar n $ and for any $t \in [n-1,n )$
$$
c^{\mathcal{J},T,(A)}_t(\sigma )<\frac{1}{1+e^{\ln t / r '}}
=\frac{1}{1+ t^{1/r '}} \leq \frac{1}{(n-1)^{1/r '}}.
$$
By setting $r' = 1/a $, we obtain \eqref{disu-c}.  We set $ p_{n} =
\frac{1}{(n-1)^{a}}$, for $n \geq \bar n $.

 Then, for $t \geq \bar n $,
$$
\sup_{\hat \sigma \in \{-1, +1 \}^V}  \mathbb{E}       (
|\mathcal{N}^-_{x }
 \cap [t, \infty )] |    \,\, | \,\, \sigma ( t ) = \hat \sigma         ) \leq
 \sup_{\hat \sigma \in \{-1, +1 \}^V}
  \sum_{n =\lfloor t \rfloor }^\infty \mathbb{E}( |\mathcal{N}^-_{x,n } |)
  $$
  \begin{equation} \label{media}
  \leq
   \sum_{n =\lfloor t \rfloor }^\infty \mathbb{E}( \sum_{\ell  =1 }^{ |\mathcal{Q}_{x,n } | }
        Y_{x,n,\ell}
   ),
\end{equation}
where $\mathcal{Q}_{x,n } = \mathcal{Q}_{x} \cap [ n-1, n )$ and
$(Y_{x,n,\ell }: \ell \in \mathbb{N}) $ are i.i.d. Bernoulli random
variables with parameter $ p_{n} $ independent from
$\mathcal{Q}_{x,n }  $. Notice that we are implicitly considering a
common probability space for all the random variables. The last
inequality in \eqref{media} comes from the graphical representation.

Using  the definition of $K$ given in \eqref{kappa},
the proof of the first item of this theorem ends by noticing that the last term in
 \eqref{media} is smaller or equal than
 \begin{equation*}\label{disu}
 \sum_{n =\lfloor t \rfloor }^\infty \mathbb{E}( |\mathcal{Q}_{x,n } |  ) \mathbb{E}( Y_{x,n,1}    )
 \leq
 K \sum_{n =\lfloor t \rfloor  }^\infty  \mathbb{E}(      Y_{x,n,1}    )=
K \sum_{n =\lfloor t \rfloor }^\infty  \frac{1}{(n-1)^{a}} ,
 \end{equation*}
the last term tends to zero for $t \to \infty$.

 \smallskip

 We give now the proof of the second item of the theorem.

Condition \eqref{cond-item2} gives 
\begin{equation}\label{fraz-2}
    \eta = \frac{1}{4d_x} \cdot \liminf_{t \to \infty } T(t)  \ln t
    >1.
\end{equation}


For $x \in V$ and $n \in \mathbb{N}$, we define the event
 $$
 F_{x,n} = \{\exists \,\ell \in \mathbb{N} : \tau_{\{x\},\ell }  \in [n-1,n) \} ,
 $$
 i.e.  there is
 at least an arrive  for the Poisson $\mathcal{P}_{\{x\}}$ in the interval $ [ n-1, n )$.
 We consider the   collection of independent and equiprobable events
   $(F_{x,n} : x \in V, \,\,\, n \in \mathbb{N}) $.
In particular, $ \mathbb{P}(F_{x,1})=1-e^{-1}$.

There exists $\bar n \in \mathbb{N}$ large enough such that, for any
$n\geq \bar n $, it results
\begin{equation*}\label{disu-c2}
    \frac{1}{n} \leq \inf \left \{
c^{\mathcal{J},T,(A)}_t(\sigma ) : t \in [n-1,n ),\, \mathcal{J} \in
\{-1,+1\}^E,\, \sigma \in \{-1,+1\}^V\right \} \qquad A \ni x  .
\end{equation*}

In fact, by \eqref{cA}, one has 
\begin{equation*}
\label{cAtheo1-2}c^{\mathcal{J},T,(A)}_t(\sigma )\geq
\frac{1}{1+e^{4d_x/T(t)}}.
\end{equation*}

Therefore, by \eqref{fraz-2}, there exists $t^\prime$ such that for
$t> t^\prime$ one has $ 4d_x/T(t)<\ln t$. Analogously to the
previous item, define $ \bar{n} = \lceil t^\prime \rceil +1$.

Hence, for any $n\geq \bar n $ and for any $t \in [n-1,n )$
$$
c^{\mathcal{J},T,(A)}_t(\sigma )>\frac{1}{1+e^{\ln t }} =\frac{1}{1+
t} \geq \frac{1}{1+n}.
$$
Then by L\'{e}vy's conditional form of the Borel-Cantelli lemma (see
\cite{Williams}), one obtains the proof of the statement.

In more details, let us consider $n \in \mathbb{N}$ and define
$\mathcal{G}_n$ as the $\sigma$-algebra generated by  all the
Poisson processes until time $n$ and the related $U$'s. Hence, the
process $\sigma (\cdot ) $ is seen as a function of the Harris'
graphical representation. Let us define the event
$$
E_n = F_{x,n} \cap \{ U_{\{ x\}, \bar \ell (n)} <
c^{\mathcal{J},T,(\{ x\})}_{ \tau_{\{ x \} ,  \bar \ell (n) }
}(\sigma (  \tau_{\{ x \} ,  \bar \ell (n) }  ) )   \} \supset
F_{x,n} \cap \{ U_{\{ x\}, \bar \ell (n)} < \frac{1}{1+n}   \} ,
$$
where $   \bar \ell (n)= \inf \{\ell : \tau_{\{ x \} , \ell } \in
[n-1, n)\} $.

Therefore, $E_n \in \mathcal{G}_n$ and $\mathbb{P} (E_n | \mathcal{G}_{n-1}) \geq  \frac{1 -
e^{-1}}{1+n} $ and this implies that
\begin{equation}\label{qpqp}
    \sum_{n=1}^\infty \mathbf{1}_{E_n} = \infty \qquad a.s.
\end{equation}
Formula \eqref{qpqp} guarantees that $S^-_{\{x\}, \infty}\cup
S^0_{\{x\}, \infty} \cup S^+_{\{x\}, \infty}$ is unbounded a.s..
Hence, $| \mathcal{N}_x| = \infty$ a.s..
\end{proof}

\begin{rem}
Consider the cubic $d$-graph $\mathbb{L}_d=(\mathbb{Z}^d,
\mathbb{E}_d)$. In this case $d_x=2d$, for any $x\in \mathbb{Z}^d$.
Moreover, given a connected finite subset $A \subset \mathbb{Z}^d$,
then $|\{e =\{x , y\}\in \mathbb{E}_d: x\in
\partial A, y \in \partial^{ext} A \}|$ is even. In this case, condition
\eqref{cond-item1} can be replaced with a less restrictive one as
follows
\begin{equation}
\label{14bis} \limsup_{t \to \infty }  T(t)  \ln t < 8.
\end{equation}
In fact, for each $\mathcal{J} \in \{-1,+1\}^E$ and $\sigma \in
\{-1,+1\}^V$, the flips in opposition of the Hamiltonian have $
\Delta_A \mathcal{H}_{\mathcal{J}}  (\sigma ) \geq 4$. This
statement can be generalized to the case of a site $x$ with even
degree $d_x$.

In particular, $d=1$ is associated to $d_x=2$, for each $x \in
\mathbb{Z}$, and also $| \{e =\{x , y\}\in \mathbb{E}_1: x\in
\partial A, y \in
\partial^{ext} A \}|=2$,  where $A \subset \mathbb{Z}$ is connected.
If there exists $ \lim_{t \to \infty }  T(t)  \ln t \not=8$, then
one among hypotheses \eqref{14bis} and \eqref{cond-item2} is
verified.
\end{rem}

As an immediate consequence of Theorem \ref{opposition} we obtain
\begin{corollary}\label{coro}
Let us consider a $d$-graph $G=(V, E)$ and a temperature profile $T$
such that $\liminf_{t \to \infty} T(t) \ln t > 4d_G $. Then the $
(k, \langle \mathcal{J} \rangle, \langle\sigma\rangle; T)$-model on
$G$ is of type $\mathcal{I}$, for any $k \in \mathbb{N}$,
$\mathcal{J} \in \{-1,+1\}^E$ and $\sigma \in \{-1,+1\}^V $.
\end{corollary}


To the benefit of the reader, we now adapt Lemma 5 in \cite{GNS} and
related definitions to our context of nonhomogeneous Markov process.
\begin{definition}
\label{def:GNS} Let us consider a continuous-time Markov process
$(Z_s : s \geq 0 )$ 
with state space
$\mathcal{Z}$. Moreover, let us take  a measurable set $A \subset
\mathcal{Z}$. The set $A$
\emph{recurs with probability} $p\in [0,1]$, if
$$
P(\{s >0 : Z_s \in A\}\,\, \text{\rm unbounded})=p.
$$
We also say that a measurable set $B \subset\mathcal{Z}^{[0,1]}$
 \emph{recurs with probability} $p\in [0,1]$, if
$$
P(\{t >0 : ( Z_{t+s} : s \in [0,1]) \in B\}\,\, \text{\rm
unbounded})=p.
$$
\end{definition}

\begin{lemma} \label{nuovoabs}
Let us consider the process $Z=(Z_t : t \geq 0 )$  and the events
$A$ and $B$ as in Definition \ref{def:GNS}. If $A$ recurs with
probability $p \in (0,1]$ and
\begin{equation}\label{estinf}
\inf\limits_{z \in A} \inf\limits_{t \geq 0} P(( Z_{t+s} : s \in
[0,1]) \in B|Z_t=z) \geq \varsigma > 0,
\end{equation}
then  $B$ recurs with probability $p' \geq p$.
\end{lemma}
\begin{proof}
Let us define  $ W=\{\{s >0 : Z_s \in A\}\,\, \text{\rm
unbounded}\}$. If  $W$ occurs we can define recursively an infinite
sequence of stopping times
$(T_j:j \geq 0)$ such that $T_0=0$ and
\begin{equation}\label{tempij}
    T_{j+1}=\inf\{t\geq T_j+1:Z_t \in A\},\,\,\text{\rm for } j  \geq 0,
\end{equation}
with the convention that $\inf \emptyset =\infty$.

Let
\begin{equation} \label{filtravodka}
\mathcal{ G}_n = \sigma\text{-algebra}( Z_t : t \leq T_{n+1}),
\end{equation}
for $n \geq 0$.

By the strong Markov property and formula \eqref{estinf}, one
obtains
\begin{equation}\label{estinf2}
P\left (  (Z_{T_j+s} : s \in [0,1]       )\in B | \mathcal{G}_{j-1}
\right ) 
\geq \varsigma  ,
\end{equation}
on $W$, for $j \in \mathbb{N}$. By \eqref{tempij} and
\eqref{filtravodka} one has $\{ (Z_{T_j+s} : s \in [0,1] ) \in B\}
\in \mathcal{G}_j$.

Formula \eqref{estinf2} gives that
$$
\sum_{j=1}^\infty P \left (     (Z_{T_j+s} : s \in [0,1]) \in B|
\mathcal{G}_{j-1} \right ) =\infty  \text{ on } W.
$$
The L\'{e}vy's extensions of the Borel-Cantelli Lemma
 implies that
$$
\sum_{j=1}^\infty \, \textbf{1}_{ \left \{ (Z_{T_j+s} : s \in [0,1]
) \in B  \right \} } =\infty \qquad P-a.s.
$$
on  $W$. Therefore $B $ recurs at least with probability $p$.
\end{proof}

Now we  introduce a new condition on the  $d$-$\mathcal{E}$ graphs.
\begin{definition} \label{stable}
Let us consider $k \in \mathbb{N}$. We say that a $d$-$\mathcal{E}$
graph $G=(V,E)$ is $k$-\emph{stable} if there exists a vertex $v \in
V $ having $d_v $ even such that both the following conditions are
satisfied
\begin{itemize}
\item if $x \in \partial^{ext}\{v\}$, then $d_x \geq 3$;
\item if $A\in \mathcal{A}_k$ is such that $v \in
\partial A \cup \partial^{ext} A$ and  $|A| \geq 2$ then
$$
|\{\{x , y\}\in E: x \in \partial A , \text{ }  y  \in \partial^{ext} A \}| - d_v 
>0 .
$$
\end{itemize}
\end{definition}
Notice that if a $d$-$\mathcal{E}$ graph is $k$-stable, then it is
$k'$-stable, for any $k'<k$, being $\mathcal{A}_{k'} \subset
\mathcal{A}_k$. We also notice that for $k=1 $ the second condition
of Definition \ref{stable} is automatically satisfied.

An example of a $k$-stable  $d$-$\mathcal{E}$ graph is the cubic
lattice $\mathbb{L}_d =(\mathbb{Z}^d,\mathbb{E}_d)$, for each $k \in
\mathbb{N}$ and $d \geq 2$.

The case $\mathbb{L}_1 =(\mathbb{Z},\mathbb{E}_1)$ is an example of
a $d$-$\mathcal{E}$ graph that is not $1$-stable, and thus it is not
$k$-stable for each $k \in \mathbb{N}$. In fact, the first condition
in Definition \ref{stable} is not satisfied.


\begin{lemma}\label{lemmafacile}
 Let $k \in \mathbb{N}$ and $G=(V,E)$ be a $k$-stable
 $d$-$\mathcal{E}$ graph. The vertex $v \in V$ is as in Definition
 \ref{stable}. Consider  $F_v =\{  f_1, \ldots , f_{d_v}\}$ as the set of incident edges to $v$,
 i.e.  $f_i \cap \{v \} = \{v \}$, for $i =1, \ldots , d_v$.
 The interactions $\mathcal{J} $ are taken as follows
\begin{equation*}\label{datesuv}
    J_e= \left \{
\begin{array}{ll}
  +1, & \text{ if } e \notin  F_v  \\
  -1, & \text{ if } e = f_i \text{ for } i = 1, \ldots , d_v /2\\
  +1, & \text{ if } e = f_i \text{ for } i = d_v /2+ 1, \ldots , d_v  \\
\end{array}
    \right .
\end{equation*}
 and the configuration $\sigma=(\sigma_u:u \in V)\in \{-1,+1\}^V$ with $\sigma_u = +1 $
 for $u \in V \setminus \{v\}$.
 Then, for any $A \in \mathcal{A}_k $ with $A \neq \{v\}$ one has $\Delta_A
\mathcal{H}_{\mathcal{J}} (\sigma ) \geq 2$.
\end{lemma}
\begin{proof}
If $A \in \mathcal{A}_k $ and $ v\notin \partial A \cup
\partial^{ext} A $, then there is nothing to prove because all the
interactions and spins involved in the computation of $\Delta_A
\mathcal{H}_{\mathcal{J}} (\sigma ) $ are equal to $+1$.

If $A = \{x\}$ with  $v \in \partial^{ext} A $ then
$$
\Delta_{\{x\}} \mathcal{H}_{\mathcal{J}}(\sigma)  = 2 \sum_{y\in
V:\{x , y\}\in E}
 J_{\{x, y\} }\sigma_x \sigma_y = 2 J_{\{x, v\} }\sigma_x \sigma_v+
 2 \sum_{ y \in V\setminus \{v\} :\{x , y\}\in E} J_{\{x, y\} }\sigma_x \sigma_y
$$
$$
= 2 J_{\{x, v\} }\sigma_x \sigma_v   + 2(d_x-1)   \geq  2(d_x-2)\geq
2 .
$$
If  $A\in \mathcal{A}_k$ is such that $v \in
\partial A \cup \partial^{ext} A$ and  $|A| \geq 2$ then
$$
\Delta_A \mathcal{H}_{\mathcal{J}}(\sigma)  = 2 \sum_{\{x , y\}\in
E: x\in \partial A, y \in \partial^{ext} A }
 J_{\{x,y\} }\sigma_x \sigma_y.
 $$
The previous sum is done on $|\{\{x , y\}\in E: x\in \partial A, y
\in \partial^{ext} A \}|$  terms. At most $d_v /2$ of theses terms
are equal to $-1$. Therefore one has
$$
\Delta_A \mathcal{H}_{\mathcal{J}}(\sigma) \geq 2  \left ( |\{\{x ,
y\}\in E: x\in \partial A, y \in \partial^{ext} A \}| -d_v\right
)\geq 2 .
$$

\end{proof}

\begin{theorem} \label{infiniti}
Let us consider  a
 $(k ,\alpha, \gamma ; T )$-model on a $k$-stable $d$-$\mathcal{E}$ graph $G=(V, E)$,
 with $ k \in \mathbb{N}$,
 $ \alpha \in (0,1)$, and  $\gamma \in [0,1]$.
 If  $T$ is fast decreasing to zero and
 positive, then  $\rho_{\mathcal{I}}>0$.
\end{theorem}
\begin{proof} \label{secondaprova}

By contradiction, we suppose that $\rho_{\mathcal{I}} =0$.

In this proof we will use several stochastic Ising models sharing
the same realized random configuration $\sigma$ but with different
interactions. Among them, we denote by \emph{original system} the
model on the $k$-stable $d$-$\mathcal{E}$ graph $G=(V, E)$ with
realized random interactions $\mathcal{J}$. In the graphical
representation we denote the Poisson processes associated to the
original model by $(\mathcal{P}_A: A \in \mathcal{A}_k)$, and the
related  uniform random variables by $(U_{A,n}: A \in \mathcal{A}_k,
n \in \mathbb{N})$.

By the assumption that the $d$-$\mathcal{E}$ graph is $k$-stable,
there exists a vertex $v \in V$ with even degree satisfying the
properties given in Definition \ref{stable}. We select such a
vertex. Let us consider the sets $B_{4k}(v)$ and $B_{8k}(v) $, where
we are using the constant $k $ given in the statement of the
theorem.

Define
\begin{equation*}\label{oranumero2}
T_{V_0} =\inf \{ t\geq 0  : \forall t' \geq t, \forall x \in V_0,
 \text{ }\sigma_x (t') = \sigma_x (t)   \} ,
 \end{equation*}
for a finite subset $V_0 \subset V$. We notice that $ T_{V_0}$ is
not a stopping time.


For $M\in \mathbb{R}^+$, let us introduce the event $F^{(1)}_M = \{
T_{B_{8k}(v) \setminus B_{4k}(v)
 } <  M\} $. It results that $F^{(1)}_{M_1} \subset F^{(1)}_{M_2}$
 when $M_1 < M_2$.

By hypothesis that $\rho_{\mathcal{I}} =0$,
\begin{equation}\label{oranumero}
\lim_{M \to \infty}\mathbb{P}(F^{(1)}_M  ) =1.
\end{equation}

%
%

Now, recall the concept of $F_v=\{f_1, \dots, f_{d_v}\}$ introduced
in Lemma \ref{lemmafacile} and consider all the finite possible
interactions $\mathcal{J}^{(1)}, \ldots, \mathcal{J}^{(q)}, \ldots,
\mathcal{J}^{(Q)} \in \{-1,+1\}^E$ obtained by taking:

%
\begin{itemize}
\item[A)] all the interactions between
two vertices $x,y  \in B_{4k}(v)$ such that $\{x,y\} \in E \setminus
\{f_1, \dots, f_{d_v/2}\}$ are equal to $+1$, while the interactions
$ J_{f_i}$, where $f_i \in F_v$ with $i =1, \ldots , d_v/2$, that
are equal to $-1$;
\item[B)] if $e = \{x,y\} \in E $ with $ x \in \partial^{ext} B_{4k}(v)$
and $ y \in \partial B_{4k}(v) $, the interaction $J_e $ is
arbitrary;
\item[C)] all the remaining interactions coincide with those of the
realized interactions $\mathcal{J}$ of the original system.
\end{itemize}
The value $Q $ is finite because we are dealing with $d$-graphs, and
it depends on $G$, $k$ and $v$.

We will construct a coupling among all the different systems with
interactions $\mathcal{J}^{(1)}, \ldots, \mathcal{J}^{(q)}, \ldots,
\mathcal{J}^{(Q)} $. 

 All the quantities related to the $q$-th system are
denoted, in a natural way, by adding, when needed, to the original
quantities the superscript $(q)$. 

Recall that all the systems described above share the initial
configuration $\sigma$ of the original one. Moreover, the collection
of Poisson processes $ (\mathcal{P}_A: A \in \mathcal{A}_k)$ are
assumed to be common for all the $Q$ systems until time $M$ and,
accordingly, the systems share also the same variables $U$'s defined
in \eqref{UAn}.

Given the $q$-th system, we define the set
$$
\mathcal{D}^{(q)}=\{x \in V : \sigma_x(M)\not= \sigma_x^{(q)}(M) \}.
$$
It is known that $|\mathcal{D}^{(q)}|<\infty$ a.s. (see e.g.
\cite{Richardson}).

We introduce the random set
$$
\mathcal{D}=\left[\bigcup_{q=1}^Q\mathcal{D}^{(q)} \right] \cup
B_{4k}(v),
$$
so that $|\mathcal{D}|<\infty$ a.s.

For any $M \in \mathbb{R^+}$ and $M'>M$, we define
$$
F^{(2)}_{M,M'} =  \bigcap_{  A \in \mathcal{A}_k : (A\cup
\partial^{ext}A) \cap \mathcal{D} \neq \emptyset  }
\{\mathcal{T}_A \cap [M,M']=\emptyset  \} ,
$$
i.e. there are no arrivals in the time interval $[M,M']$ for the
Poisson processes $\mathcal{P}_A$ labeled by $A \in \mathcal{A}_k $
with $ (A\cup
\partial^{ext}A) \cap \mathcal{D} \neq \emptyset $. In this case one has
$F^{(2)}_{M,M_1'} \supset F^{(2)}_{M,M_2'}$ when $ M < M_1' <M_2'$.
Furthermore,
\begin{equation}
\label{M'} \lim_{M' \to M^+} \mathbb{P}( F^{(2)}_{M,M'}  )=1.
\end{equation}
Notice that the probability in \eqref{M'} depends only on the
difference $(M' - M)$.

Among all the different systems --described in A), B) and C)-- we
select the only one having interactions such that $J_e
=\sigma_{x}(M) $ when $e = \{x,y\} \in E $ with $ x \in
\partial^{ext}B_{4k}(v)$ and $ y \in
\partial B_{4k}(v) $. We denote such random interactions as
$\bar{\mathcal{J}}$. We call the corresponding system as
\emph{capital system}, this process is denoted by $(\Sigma(t): t
\geq 0)$. We assume that it
 is a stochastic Ising model, and we will
check it below.

We denote by $\bar{\mathcal{S}}_{A,t}^-$ (resp.
$\bar{\mathcal{N}}_{x}^-$) the random set in \eqref{SAt-} (resp. in
\eqref{N-}), by replacing $\sigma$ with $\Sigma$.


For $M \in \mathbb{R}^+$, we define
$$
F^{(3)}_{M} =\bigcap_{  A \in \mathcal{A}_k : A \subset B_{4k} (v) }
\{ \bar{\mathcal{S}}_{A,\infty}^- \cap [M, \infty)= \emptyset  \} .
$$
One has that $F^{(3)}_{M_1} \subset F^{(3)}_{M_2}$
 when $M_1 < M_2$.

Observe that
$$
\bigcap_{  x \in B_{4k} (v) } \{ \bar{\mathcal{N}}_{x}^- \cap [M,
\infty)= \emptyset  \} \subset F^{(3)}_{M}.
$$
Since $(\Sigma(t): t \geq 0)$ will be proved to be a stochastic
Ising model, we can use \eqref{eq:thm1new} in Lemma \ref{serve} in
order to obtain
\begin{equation}\label{F3M}
    \lim_{M \to \infty } \mathbb{P} ( F^{(3)}_{M}) =1.
\end{equation}

By \eqref{oranumero} and \eqref{F3M} one can select $M$ large enough
to have that $\mathbb{P} ( F^{(i)}_{M})> 3/4 $, for $i=1,3$. Now, by
\eqref{M'}, one can take $M' $ close enough to $M$ so that
$\mathbb{P}( F^{(2)}_{M,M'}  )> 3/4$. Therefore, under such choices
of $M$ and $M'$, one has
$$
\mathbb{P}(F^{(1)}_{M} \cap F^{(2)}_{M,M'} \cap F^{(3)}_{M'} )\geq
\mathbb{P}(F^{(1)}_{M} \cap F^{(2)}_{M,M'} \cap F^{(3)}_{M} )
>\frac{1}{4}.
$$

We define $ (\mathcal{\bar P}_{A}: A \in \mathcal{A}_k) $ a set of
Poisson processes with rate $1$ that are active in the time interval
$[M, M']$ and 
independent from all the other Poisson processes and also from the
$U$'s of the original system.

We call $ \bar{ \mathcal{T}}_{A} =( {\bar \tau}_{A,n} \in [M,M']: n
\in \mathbb{N})$ the set of arrivals of the Poisson process
$\mathcal{\bar P}_{A}$. In order to use Harris' graphical
representation, we define a set of uniform random variables $ ({\bar
U}_{A, n}:  n \in \mathbb{N}) $ associated to $\bar{
\mathcal{T}}_{A}$, for any $A \in \mathcal{A}_k$. The ${\bar U}$'s
are independent from all the Poisson processes and from the $U$'s of
the original system.

Given $A \in \mathcal{A}_k$, we consider a new process $
\mathcal{P}^\star_{A}$ whose set of arrivals $\mathcal{T}^\star_{A}$
is defined as follows:
\begin{equation}
\label{Pstar} \mathcal{T}^\star_{A}= \left\{%
\begin{array}{ll}
   \{\mathcal{T}_{A} \cap [0,M]\}\cup \mathcal{\bar T}_{A} \cup \{\mathcal{T}_{A} \cap  [M',\infty)\}, & \hbox{if $A=\{x\}$ and $x \in \mathcal{D}$;} \\
\mathcal{T}_{A} , & \hbox{otherwise.}
\end{array}%
\right.
\end{equation}
For each $A\in\mathcal{A}_k$, the set $\mathcal{T}^\star_{A}$ is
distributed as a Poisson process of rate 1, and $
(\mathcal{P}^\star_{A}: A \in \mathcal{A}_k) $ is a set of
independent Poisson processes with rate $1$, which are also
independent from the $U$'s and from the $\bar{U}$'s.

The $U$'s (resp. $\bar{U}$'s) are used when the arrivals of the
Poisson processes are taken from $\mathcal{T}_{A}$ (resp. from
$\mathcal{\bar T}_{A}$).

All these random quantities are employed to construct the capital
system $(\Sigma(t):t \geq 0)$ via the graphical representation.
Since $(\mathcal{P}^\star_{A}: A \in \mathcal{A}_k) $ is a
collection of independent Poisson processes and by definition of the
$U$'s and the $\bar{U}$'s, the capital system is a stochastic Ising
model.

We set the constant
$$
\eta_{M, M'} =\inf_{t \in [M,M']} \inf_{x \in V} \inf_{\sigma \in
\{-1,+1\}^V} \inf_{\mathcal{J} \in \{-1,+1\}^E} \{ c^{\mathcal{J},
T,(\{x\})}_t (\sigma ) , 1-  c^{\mathcal{J}, T,(\{x\})}_t (\sigma )
\}.
$$
We notice that $\eta_{M, M'}>0$. In fact, the temperature profile
$T$ is positive in $[M,M']$ and $t $ ranges in a compact interval.
Moreover, $x$ could be taken only in $Cell$ for the periodicity of
the $d$-graph, and this implies also that $\mathcal{J}$ and $\sigma$
could be considered only in a finite region.

We define the events
$$
F^{(4)}_{M,M'} = \bigcap_{x \in \mathcal {D} \setminus B_{4k}(v)
}\Big[ \{\bar{ \mathcal{T}}_{\{x\}}= \{ \bar \tau_{\{x\}, 1}\}\}
$$$$
\cap\{ \bar U_{\{x\},1} \mathbf{1}_{\{\Sigma_x(M)
\not=\sigma_x(M)\}}<\eta_{M, M'} \} \cap \{ \bar U_{\{x\},1} \geq
\mathbf{1}_{\{\Sigma_x(M) =\sigma_x(M)\}}(1-\eta_{M, M'}) \}\Big]
$$
and
$$
F^{(5)}_{M,M'} = \bigcap_{x \in B_{4k}(v) }\Big[ \{\bar{
\mathcal{T}}_{\{x\}}= \{ \bar \tau_{\{x\}, 1}\}\}
$$
$$
\cap\{ \bar U_{\{x\},1} \mathbf{1}_{\{\Sigma_x(M) =-1\}}<\eta_{M,
M'} \} \cap \{ \bar U_{\{x\},1} \geq \mathbf{1}_{\{\Sigma_x(M)
=+1\}}(1-\eta_{M, M'}) \}\Big].
$$
By the independence properties of the considered random quantities,
it results
$$
\mathbb{P}\left(F^{(4)}_{M,M'}\cap F^{(5)}_{M,M'}\, \Big| \,
F^{(1)}_{M} \cap F^{(2)}_{M,M'}\cap F^{(3)}_{M'}\right)
$$
$$
\geq \sum_{\ell=0}^\infty
\left[(M'-M)e^{-(M'-M)}\eta_{M,M'}\right]^\ell \cdot
\mathbb{P}\left(|\mathcal{D}|=\ell \,\Big|\,F^{(1)}_{M} \cap
F^{(2)}_{M,M'}\cap F^{(3)}_{M'}\right)>0,
$$
and this gives $ \mathbb{P}\left(F^{(1)}_{M} \cap F^{(2)}_{M,M'}\cap
F^{(3)}_{M'} \cap F^{(4)}_{M,M'}\cap F^{(5)}_{M,M'}\right)>0 $.

\medskip

Assume from now on that $F^{(1)}_{M} \cap F^{(2)}_{M,M'}\cap
F^{(3)}_{M'} \cap F^{(4)}_{M,M'}\cap F^{(5)}_{M,M'}$ occurs.

We claim that $\Sigma_u(M')=+1$, for each $u \in B_{4k}(v)$. In
particular, on $F^{(2)}_{M,M'} \cap F^{(5)}_{M,M'}$, for $u \in
B_{4k}(v) $ one has the occurrence of only one arrival in $[M, M']$
of the Poisson processes involving $u$, and such an arrival is
$\bar{\tau}_{\{u\},1}$. On $F^{(5)}_{M,M'}$ this arrival generates a
flip at $u$ for the capital system if and only if $\Sigma_u(M)=-1$.

We also claim that $\Sigma_u(M')= \sigma_u(M')$ for each $u \notin
B_{4k}(v)$. We prove it by separating two subcases: (i) $u \in
\mathcal{D} \setminus B_{4k}(v) $ and  (ii) $u \notin \mathcal{D} $.

(i) On $F^{(2)}_{M,M'} \cap F^{(4)}_{M,M'}$ for $u \in \mathcal{D}
\setminus B_{4k}(v) $ there is only one arrival in $[M, M']$ of the
Poisson processes involving $u$, and the arrival is
$\bar{\tau}_{\{u\},1}$. On $F^{(4)}_{M,M'}$ this arrival generates a
flip at $u$ for the capital system if and only if $\Sigma_u(M)\not=
\sigma_u(M)$.

(ii) On $F^{(2)}_{M,M'} $ one knows that the arrivals in the
interval $[M, M']$ involving the sites outside $ \mathcal{D}$ are
the same for the original and the capital system. Moreover, the
original and the capital system share the same configuration outside
$\mathcal{D}$ and the same $U$'s. Therefore, such systems have
identical flips outside $ \mathcal{D}$. Since $\Sigma_u(M)=
\sigma_u(M)$ for $u \notin \mathcal{D}$ then $\Sigma_u(t)=
\sigma_u(t)$ for $u \notin \mathcal{D}$ and $t \in[M,M']$.

\medskip

We now need to check the following:
\begin{itemize}
\item[$(\mathbf{H})$]
for any $t \geq M'$ we have $ \Sigma_u (t) = \sigma_u (t) $ when $u
\notin B_{4k} (v)$ and $ \Sigma_u (t) = +1 $ when $u \in B_{4k}
(v)\setminus \{v\}$.
\end{itemize}

We can define recursively the sequence of random times $(\psi_\ell :
\ell \in \mathbb{N}) $. Let
$$
\psi_\ell = \inf \{ \tau_{A,n}> \psi_{\ell-1} :A\in \mathcal{A}_k
\text{ with }(A \cup
\partial^{ext} A) \cap B_{4k} (v) \neq \emptyset , n \in \mathbb{N}\},
$$
for $\ell \geq 1$ and the conventional agreement that $\psi_0=M'$.

For each $\ell \geq 1$ one has that the random set
$$
\Upsilon_\ell =\{ \tau_{A,n} \in  [ \psi_{\ell-1} , \psi_{\ell}]:
A\in \mathcal{A}_k \text{ with }(A \cup
\partial^{ext} A) \cap B_{4k} (v) = \emptyset , n \in \mathbb{N} \}
$$
has infinite cardinality a.s.

In the interval $ [M',\psi_{1} )$ we are dealing only with arrivals
having $A \in \mathcal{A}_k $ such that $(A \cup
\partial^{ext} A) \cap B_{4k} (v) = \emptyset $, hence belonging to $\Upsilon_1$. Therefore,
 by construction and from the fact that
$\Sigma_u(M')= \sigma_u(M')$ for each $u \notin B_{4k}(v)$, the
original and the capital system share the same flips in
$\Upsilon_1$. This leads to 
$ \Sigma_u (t) = \sigma_u (t) $ when $u \notin B_{4k} (v)$ and $
\Sigma_u (t) = +1 $ when $u \in B_{4k} (v)\setminus \{v\}$, for $t
\in [M',\psi_{1} ]$.

Now we analyze the arrival in $\psi_1$. There exist unique (a.s.) $A
\in \mathcal{A}_k$ and $n \in \mathbb{N}$ such that the arrival
$\tau_{A,n}=\psi_1$. Four cases can be considered for the set $A$.
They are presented and discussed below.
\begin{itemize}
    \item $A=\{v\}$.

    In this case a flip at $A$ influences only $\Sigma_v(\psi_1^+)$.
    Therefore $ \Sigma_u (\psi_{1}^+) = \sigma_u (\psi_{1}^+) $
when $u \notin B_{4k} (v)$ and $ \Sigma_u (\psi_{1}^+) = +1 $ when
$u \in B_{4k} (v)\setminus \{v\}$.
    \item $(A \cup \partial^{ext}A) \subset B_{4k}(v)$ and $A \not=
    \{v\}$.

By Lemma \ref{lemmafacile} we know that $\Delta_A
\mathcal{H}_{\bar{\mathcal{J}}}(\Sigma(\psi_1)) \geq 2$. However,
the occurrence of $F^{(3)}_{M'}$ guarantees that this arrival does
not correspond to a flip.

    \item $A \subset B_{4k}(v)$ and $\partial^{ext}A \not \subset B_{4k}(v)$.

By the selected interactions $\bar{\mathcal{J}}$ for the capital
system (in particular: if $e = \{x,y\} \in E $ with $ x \in
\partial^{ext}B_{4k}(v)$ and $ y \in
\partial B_{4k}(v) $,   then $J_e =\sigma_x(M)=\sigma_{x}(M')=\Sigma_x(M') $) we are in the position of applying Lemma
\ref{lemmafacile}. Analogously to the previous case, $F^{(3)}_{M'}$
assures that there is not a flip at $A$ in $\psi_1$.

    \item $A \not \subset B_{4k}(v)$ and $(A \cup \partial^{ext}A) \cap B_{4k}(v) \not=\emptyset$.

By the selected interactions $\bar{\mathcal{J}}$ one has
$$
\Delta_A \mathcal{H}_{\bar{\mathcal{J}}}(\Sigma(\psi_1)) \geq
\Delta_A \mathcal{H}_{\mathcal{J}}(\sigma(\psi_1)).
$$
Formula \eqref{cA} says that
\begin{equation}
\label{CA2} c^{\bar{\mathcal{J}}, T,(A)}_{\psi_1}(\Sigma(\psi_1) )
\leq c^{\mathcal{J}, T,(A)}_{\psi_1}(\sigma(\psi_1) ).
\end{equation}
The occurrence of $F^{(1)}_M$ assures that there are no flips at $A$
after time $M$ for the original system. Hence, since the original
and the capital system share the same $U$'s, \eqref{CA2} states that
there is not a flip at $A$ in $\psi_1$ for the capital system.
\end{itemize}

By iterating the arguments above, we obtain $(\mathbf{H})$. In
particular, the event
$$
\bar{A}=\{ \text{all the spins in } B_{4k}(v) \setminus \{v\} \text{
are equal to } +1      \}
$$
recurs with probability  $ p \geq \mathbb{P}\left(F^{(1)}_{M} \cap
F^{(2)}_{M,M'}\cap F^{(3)}_{M'} \cap F^{(4)}_{M,M'}\cap
F^{(5)}_{M,M'}\right)>0 $.

When $\Sigma (t ) \in \bar{A}$ there is a positive probability that
there is an arrival for the Poisson $\mathcal{P}_{\{v\}} $ in the
interval $[t, t+1]$. Such an arrival has probability $1/2$ to
generate a flip at $\{v\}$, being $\Delta_{\{v\}}
\mathcal{H}_{\bar{\mathcal{J}}}(\Sigma(t))=0$ for each $t \geq M'$.
Thus,  by Lemma \ref{nuovoabs}, the event
$$
\{  \Sigma_v (\cdot ) \text{ has a flip in }[0,1]    \}
$$
recurs with positive probability greater than or equal to $p$.
\medskip

Re-sampling the interactions of the original system only over the
finite set
 $$
\{ e=\{u,w\} \in E:  \{ u,w\} \cap B_{4k} (v)\neq \emptyset \}
 $$
there is a positive probability that it coincides with the
interactions of the capital
system, since $\alpha \in (0,1)$. 
Therefore, also in the original one, the probability that
$|\mathcal{N}_v| = \infty$
is positive. Hence, by ergodicity,
$\rho_{\mathcal{I}}$ must be larger than zero. This concludes the
proof.
\end{proof}

\begin{rem}
We notice that, the previous theorem can be generalized by
considering  the parameter $\alpha \in (0,1)$ but giving the initial
configuration $\sigma\in \{-1,+1\}^V$. In this case we can only
prove, following the same line of the previous proof, that there
exists a positive frequency of vertices that flip infinitely often.
Hence, in general, we only know that
$$
\liminf_{\ell \to \infty } \frac{| \{ x \in  B_\ell (O) :  x \text{ flips
finitely many times} \} |}{ | B_\ell (O) | } >0,
$$
but the limsup could not coincide with the liminf  and they could
depend on the initial configuration $\sigma $. We can conclude that
for any initial configuration $\sigma$ the model is of type $
\mathcal{M}$ or $\mathcal{I}$. This means that there is not
convergence of the process  almost surely.
\end{rem}

We now present a theorem in which we explore the possibility that
$\rho_{\mathcal{I}} $ is larger than zero without requiring the
positivity of the temperature profile. The basic assumption is that
some specific sites fixate.   We address the reader to the next
section for sufficient conditions to obtain $\rho_{\mathcal{F}} >0
$.

\begin{theorem}\label{unione}
Let us consider a $ (k,  \alpha , \gamma; T)$-model on a $d$-graph
$G= (V, E)$ with $k \in \mathbb{N}$, $\alpha \in
    (0,1)$, $\gamma \in [0,1]$ and $T $ fast decreasing to zero. Suppose that there
    exist finite  sets $R, R'\subset V $, a finite set
    $S \subset E$, $\hat {\sigma}_R =(\hat \sigma_u: u \in R) \in  \{-1,+1\}^R$, and
    $\hat {\mathcal{J}}_S = (\hat J_e: e \in S)\in
    \{-1,+1\}^S$ with the following properties
    \begin{itemize}
        \item[(i)]
\begin{equation}
\label{thm4new1} \mathbb{P} \left (\bigcap_{u \in R}\{\lim_{t \to
\infty} \sigma_u (t)=\hat{\sigma}_u\}   \big |   \bigcap_{e \in S}
\{ J_e = \hat J_e\} \right  )>0.
\end{equation}

        \item[(ii)] For any $\mathcal {J} $ coinciding with $ \hat
        {\mathcal{J}}_S$ on $S$ and for any  $\sigma \in \{-1,+1\}^V$ coinciding with
$\hat \sigma_R$ on $R$ there exists  $D \in
        \mathcal{A}_k $ such that $D \subset R'$, with
        $\Delta_D \mathcal{H}_{\mathcal{J}} (\sigma ) \leq 0
        $.
    \end{itemize}
Then, the model is of type $\mathcal{M}$.
\end{theorem}

\begin{proof}
First of all, we notice that there exists $\mathcal{J} \in
\{-1,+1\}^E$ coinciding with $ \hat {\mathcal{J}}_S $ on $S$ with
positive probability, since $S$ is finite and $\alpha \in (0,1)$.

By \eqref{thm4new1} and by  ergodicity one has  that
$\rho_{\mathcal{F}}>0$. So, it is sufficient to prove that
$\rho_{\mathcal{I}}>0$.


In the following we use the probability measure $P(\cdot ) =
\mathbb{P}(\cdot |   \bigcap_{e \in S} \{ J_e = \hat J_e\} ) $.

 We define the set
$$
\hat A=\{\sigma \in \{-1,+1\}^V : \sigma_u =\hat{\sigma}_u,
\forall\,u \in R\}.
$$
By definition of $\hat \sigma_R$, we know that $\hat A$ recurs with
positive probability. The set $D$ in the statement of the theorem
depends on $\sigma$. We write
$$
 A_D = \{\sigma \in \{-1,+1\}^V : \sigma_u =\hat{\sigma}_u, \forall\,u
\in R ,\quad   \Delta_D \mathcal{H}_{\mathcal{J}} (\sigma ) \leq 0
\},
$$
for $ D \in \mathcal{A}_k  $ such that $D \subset R'$. Thus, $\hat A
= \bigcup_{D \in \mathcal{A}_k : D \subset R' }   A_D$ is  a finite
union. Therefore, one can select a $\bar D \in \mathcal{A}_k $ with
$ \bar D \subset R' $ such that $A_{\bar D} $ recurs with positive
probability.

Now, the set $B$, in the frame of Lemma \ref{nuovoabs},
is taken as
$$
B=\{\mathcal{S}^+_{ \bar D ,1} \cup\mathcal{S}^0_{ \bar D,1} \neq
\emptyset \} .
$$

 Then
$$
\inf\limits_{\sigma \in A_{\bar D}} \inf\limits_{t \geq 0} P(
(\sigma(t+s): s \in [0,1])    \in B | \sigma(t)=\sigma)
$$
\begin{equation} \label{bobo}
\geq \frac{1}{2} P( \{  \mathcal{T}_{\bar D} \cap  [t, t +1] \neq
\emptyset \}\cap \{  \mathcal{T}_{A} \cap [t, t +1] = \emptyset :
\forall \, A  \in \mathcal{A}_k \text{ with } A\neq \bar D,  A \cap
(R' \cup
\partial^{ext} R ')\neq \emptyset \} ),
\end{equation}
where $1/2$, in the previous formula, is due to $\Delta_{\bar D}
\mathcal{H}_{\mathcal{J}} (\sigma ) \leq 0 $.

By the independence of the Poisson processes, \eqref{bobo} is larger
than  or equal to
$$
 \frac{1}{2}  (1-e^{-1}) e^{-K |R'\cup
\partial^{ext} R'  |} >0,
$$
where $K$ is given in \eqref{kappa}.

 Lemma \ref{nuovoabs}
guarantees that $B$ recurs with positive probability. Thus, the
ergodic theorem assures that $\rho_{\mathcal{I}}>0$.
\end{proof}

\begin{example}\label{commento}
In Figure \ref{fig1} we provide an example. It is here represented a
region of the $d$-graph $G$, with $d=2$. For the notation, see
Theorem \ref{unione}.
  Consider $k=1$, $\alpha \in (0,1)$, $\gamma \in (0,1]$
(resp. $\gamma \in [0,1)$) and $T$ fast decreasing to zero. Set the
black bullets as sites with spin $+1$ (resp. $-1$) and the other
spins can be arbitrarily selected.
  The set $R$ is formed by the black bullets.
The set $S$ is given by all the edges connecting a black bullet with
another black bullet or with the white one. The continuous line
segment represents an edge $e$ with $J_e=+1$, while the dotted line
segment is associated to $J_e=-1$. With positive probability the
spins over $R$ do not flip, and remain constant at $+1$ (resp. $-1$)
if initially taken positive (resp. negative) (see \eqref{fa} in
Lemma \ref{serve}). Assume that this is the case.  $R'$ is formed by
the cental white bullet. So $D = R'$ and $\Delta_D
\mathcal{H}_{\mathcal{J}} (\sigma) =0$. Hence,
$\rho_{\mathcal{I}}>0$. Notice that the only interactions which
really matter are those in $S$, while the other ones can be
arbitrarily selected.
\end{example}

\begin{figure}
  \includegraphics[scale=0.4]{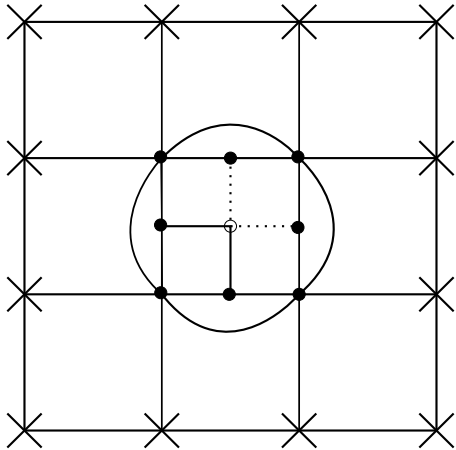}
  \caption{\small{Example of $2$-graph in which Theorem \ref{unione}
  can be applied. Such a graph can be also used as an example of
  Theorem \ref{cultb}.
   }}
  \label{fig1}
\end{figure}

We finish the section with a particular case in which we show that
the parameter $k$ plays a role to establish the class of a $(k,
\alpha, \gamma;T)$-model. It is known that the $(1, \alpha,
\gamma;T)$-model on the hexagonal lattice is of type $\mathcal{F}$
when $T\equiv 0$ (see \cite{NNS00}) and in the next section this
result is proven when $T$ is fast decreasing to zero (see Theorem
\ref{culta}). It is important to notice that the hexagonal lattice
in our framework is a $2$-graph.

In the following result we show that a $(k, 1, 1/2;T)$-model with $k
\geq 2 $ cannot belong to $\mathcal{F}$.  The question if it is of
kind $\mathcal{M}$ or $\mathcal{I}$ remains open.

\begin{theorem}\label{hexag}
The $(k, 1, 1/2;T)$-model on the hexagonal $2$-graph $G_H=(V_H,E_H)$
with $k \geq 2$ is not of type $\mathcal{F}$, for any temperature
profile $T$.
\end{theorem}
\begin{proof}\label{proofhexag}

By contradiction suppose that the model is of type $\mathcal{F}$.
Therefore there exists, for any $x \in V_H$,
$$
\sigma_x (\infty ) =\lim_{t \to \infty } \sigma_x (t ) \,\,\,\,\,
a.s.
$$
We denote with $ \sigma (\infty ) = ( \sigma_x (\infty ) : x \in
V_H)$ the random limit configuration. By ergodicity and since $\gamma
=1/2$  
we obtain that
\begin{equation}\label{limiteesag}
    \lim_{\ell \to \infty} \frac{|  \{x\in B_\ell (v): \sigma_x (\infty ) =+1
\} |}{ |B_\ell (v)| } = \frac{1}{2}  \,\,\,\,\, a.s.
\end{equation}
for any $v \in V_H$.

As in \cite{NNS00} we define the domain wall $\mathcal{D}^H$ as a
subset of the hexagonal dual lattice, which separates the positive
spins from the negative ones by joining the center of the hexagons
(see Figure \ref{fighexag}). By \eqref{limiteesag} $\mathcal{D}^H
\neq \emptyset $  almost surely. The dual lattice is formed by three
classes of edges: the vertical edges, the ascendent ones and the
descendent ones (see Figure \ref{fighexag}). As in
\eqref{limiteesag}, there exists the density $\rho_{ver}$
($\rho_{asc}$, $\rho_{des}$, resp.) of the vertical (ascendent,
descendent, resp.) edges of the domain wall. By symmetry and being
$\mathcal{D}^H \neq \emptyset$, we have that
$\rho_{ver}=\rho_{asc}=\rho_{des}>0$ almost surely.

\begin{figure}[!h]
  \includegraphics[scale=0.5]{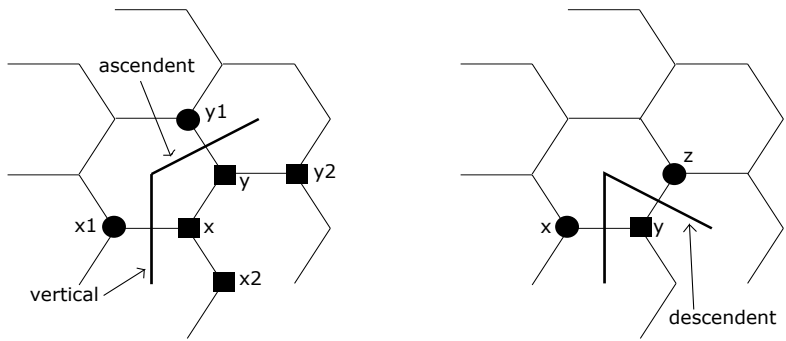}
  \caption{\small{
  The hexagonal lattice with the domain wall $\mathcal{D}^H$ in
  bold. The squares and the bullets represents spins with opposite
  signs in the two situations of $2\pi/3$ (left panel) and $\pi/3$ (right panel) angles.
 }}
  \label{fighexag}
\end{figure}

On each site $v \in V_H$ with $\sigma_v=-1$ we consider the triangle
$\Phi_{v,H}$ connecting the centers of the three hexagons containing
$v$. Such triangle is a closed convex set, in that it is given by
its internal part and its boundary. We introduce the set containing
all such triangles
$$
\Phi_{H}=\bigcup_{v\in V_H:\sigma_v=-1}\Phi_{v,H}.
$$
The set $\mathcal{D}^H$ is the boundary of the set $\Phi_{H}$.
Simplicial homology arguments (see \cite{spanier}) gives that the
domain wall $\mathcal{D}^H$ can be decomposed in connected
components, each of them belonging to one of the following families:
\begin{itemize}
\item \emph{Closed component:} the component is given by a sequence of
edges  forming a cycle.
\item \emph{Bidirectional infinite component:} the component is
not closed and each edge belonging to the component has two adjacent
edges; moreover, such adjacent edges are disjoint.
\end{itemize}

Two adjacent edges of the set $\mathcal{D}^H$ may form an angle of
$\pi$, $2\pi/3$ and $\pi/3$.

If there exists a closed component, then there exists at least one
couple of adjacent edges forming an angle different from $\pi$. In
this case, by ergodicity, the sum of the densities of angles
$2\pi/3$ and $\pi/3$ is different from zero.

If all the components are of type bidirectional infinite, then we
have two subcases. One of the components has an angle different from
$\pi$ or, otherwise, the components are  straight lines (for
example, with all vertical edges). In the case of components with an
angle different from $\pi$, ergodicity leads that, as in the case of
existence of a closed component, the sum of the densities of angles
$2\pi/3$ and $\pi/3$ is different from zero. The straight line case
has zero probability to occur. In fact, ergodicity and symmetry
guarantee the existence of another component of $\mathcal{D}^H$
which is a straight line as well and it is not parallel to the
vertical one. The intersection of the two components forms an angle
different from $\pi$, and we fall in the previous case. However, we
notice that the coexistence of two nonparallel straight lines is
absurd, in that one of the  spins at the four corners formed by the
intersection of the components should have simultaneously value $+1$
and $-1$.

 If the density of the angles $\pi/3$ is different from zero
(see Figure \ref{fighexag}, right panel), then we have two edges
$\{x,y\}, \{y,z\} \in E_H$ such that the random time
$$T_{\{x,y,z\}}=\inf \{ t \in \mathbb{R}_+ : \forall t' \geq t,
\,\sigma_x (t') = \sigma_z (t')=-\sigma_y (t') \}$$ is finite a.s..
However, there are infinite arrivals of the Poisson
$\mathcal{P}_{\{y\}}$ in the random interval $(T_{\{x,y,z\}},
\infty)$ and at each arrival the site $y$ flips with a probability
at least $1/2$ (the value of such a probability depends on the
temperature at the arrival time, and it is one in the case of
zero-temperature). By Lemma \ref{nuovoabs},  
we have a
contradiction and the $(k, 1, 1/2;T)$-model is of type $\mathcal{M}$
or $\mathcal{I}$.

If the density of the angles $\pi/3$ is zero,  then the density of
the angles of $2\pi/3$ is different from zero (see Figure
\ref{fighexag}, left panel). In this case there exists an edge
$\{x,y\} \in E_H$ and its four adjacent edges $\{x_1,x\}, \{x_2,x\},
\{y_1,y\}, \{y_2,y\} \in E_H$ such that the random time

$$
T_{\{x,y,x_1,x_2,y_1,y_2\}}=$$$$=\inf \{ t \in \mathbb{R}_+ : \forall t'
\geq t, \,\sigma_{x_1} (t') = \sigma_{y_1} (t') = -\sigma_{x_2} (t')
= -\sigma_{y_2} (t') \text{ and } \sigma_{x} (t') = \sigma_{y} (t')
\}
$$
is finite a.s. Analogously to the previous case, there are infinite
arrivals of the Poisson $\mathcal{P}_{\{x,y\}}$ in
$(T_{\{x,y,x_1,x_2,y_1,y_2\}}, \infty)$ and at each arrival the set
$\{x,y\}$ flips with a probability equals to $1/2$.  Lemma
\ref{nuovoabs} leads to a contradiction, therefore the $(k, 1,
1/2;T)$-model cannot be of type $\mathcal{F}$.
\end{proof}

%


\section{Conditions for $\rho_{\mathcal{F}}>0$ }
\label{rhoFno0}


In this section we give some sufficient conditions to obtain
$\rho_{\mathcal{F}}>0$. 
Let us consider a $d$-graph $G=(V, E)$. For a finite set $\Lambda
\subset V$, we define
$$
E(\Lambda)=\{\{u,v \} \in E : u,v \in \Lambda\}      .
$$
For  $\mathcal{J} \in \{-1,+1\}^E$, $\sigma \in  \{-1,+1\}^V$, we
also set
 \begin{equation} \label{finhami2}
\mathcal{H}_{\Lambda, \mathcal{J}}(\sigma)  =
 -\sum_{e=\{x,y\} \in E( \Lambda )  } J_{e} \sigma_x  \sigma_y
\end{equation}
and 
\begin{equation} \label{densita}
D_{\Lambda, \mathcal{J}}(\sigma) =\frac{ \mathcal{H}_{\Lambda,
\mathcal{J}}(\sigma)}{ |\Lambda |} .
\end{equation}

%

We are now ready to present the following
\begin{lemma}\label{Lemrichiesto}
Let us consider a $d$-graph $G=(V,E)$.     
For any $\mathcal{J } \in \{-1,+1\}^E$, $ \sigma  \in \{-1,+1\}^V$
and any finite set $\Lambda \subset V$ one has
\begin{equation*}\label{ddi}
 D_{\Lambda,
\mathcal{J}}(\sigma )  \in [-d_G/2 , d_G/2 ],
\end{equation*}
where $d_G$ is the maximal degree of the graph $G$.
\end{lemma}
\begin{proof}
 We rewrite the expression in \eqref{finhami2} as
 $$
\mathcal{H}_{   \Lambda  , \mathcal{J}}(\sigma) = -
\frac{1}{2}\sum_{x \in \Lambda} \left [ \sum_{ \{x, y\} \in E(
\Lambda )  } J_{\{x, y\}} \sigma_x \sigma_y       \right ] .
 $$
The term inside the square brackets belongs to $ [-d_G , d_G] $. In
fact, for any $x \in \Lambda$, by definition we have $d_x \leq d_G$
and $J_{\{x, y\}}, \, \sigma_x, \, \sigma_y \in \{ -1, +1\}$. Thus
$$
\mathcal{H}_{\Lambda   , \mathcal{J}}(\sigma) \in \left [
 - \frac{d_G |\Lambda |}{2} , \frac{d_G |\Lambda|}{2} \right  ] .
$$
By \eqref{densita}, we have the thesis.
%
\end{proof}%

We give two more definitions that will be used in the proof of the
next lemma. They are based on the embeddedness of the $d$-graph and
on the concept of cell. For a $d$-graph $G=(V,E)$ and $\ell \in
\mathbb{N}$ we define
\begin{equation*}
\label{RegA-Lemma0} \mathcal{Z}_\ell=\left\{A\in \mathcal{A}_k : A
\cap [-\ell,\ell)^d \neq \emptyset \text{ and } (A \cup
\partial^{ext} A ) \not \subset [-\ell,\ell)^d
\right\}
\end{equation*}
and
\begin{equation*}
\label{RegA-Lemma+} \mathcal{Z}_\ell^+=\left\{A\in \mathcal{A}_k :
(A \cup
\partial^{ext} A) \subset [-\ell,\ell)^d 
\right\}.
\end{equation*}
Thus $\mathcal{Z}_\ell \cap \mathcal{Z}_\ell^+= \emptyset$.
By definition, there exists a constant $ M_{k,G}>0 $, depending on
$k $ and the graph $G$, such that  $ |\mathcal{Z}_\ell | \leq
M_{k,G} \ell^{d-1} $ for any $\ell \in \mathbb{N}$. Analogously,
there exists a constant $ M^+_{k,G}>0 $ such that $
|\mathcal{Z}^+_\ell | \leq M^+_{k,G} \ell^{d} $ for any $\ell \in
\mathbb{N}$.





\begin{lemma}\label{Lemnonrichiesto}
Let us consider a $(k, \alpha,\gamma; T)$-model on a $d$-graph
$G=(V,E)$ with $k \in \mathbb{N}$, $\alpha \in [0,1]$, $\gamma \in
[0,1]$ and temperature profile $T$ fast decreasing to zero. Then
$\mathbb{E}[ |\mathcal{S}^+_{A,\infty}|]<\infty $ for any $A \in
\mathcal{A}_k$.
\end{lemma}
\begin{proof}
Let $\ell \in \mathbb{N}$. We denote the quantity $\mathcal{H}_{ V
\cap [-\ell,\ell)^d , \mathcal{J}}( \cdot)$ by $\mathcal{H}_{\ell ,
\mathcal{J}}( \cdot) $.


We claim that there exists the following  bound for $\mathcal{H}_{
\ell , \mathcal{J}}(\sigma (t ))$
\begin{equation}
\label{disuH} \mathcal{H}_{\ell  , \mathcal{J}}(\sigma (t )) -
\mathcal{H}_{\ell , \mathcal{J}}(\sigma (0 )) \leq \sum_{A\in
\mathcal{A}_k : A \cap [-\ell,\ell)^d  \neq \emptyset }
kd_G|\mathcal{S}^-_{A,t}|+\sum_{A\in \mathcal{Z}_\ell }
kd_G|\mathcal{S}^-_{A,t} \cup \mathcal{S}^0_{A,t} \cup
\mathcal{S}^+_{A,t}|-2 \sum_{A\in \mathcal{Z}_\ell^+}
|\mathcal{S}^+_{A,t}|,
\end{equation}
where $kd_G$ represents an upper bound for $\Delta_A \mathcal{H}_{
\mathcal{J}    }         (\sigma )$, for each $\sigma \in \{-1,
+1\}^V$ and $A \in \mathcal{A}_k$ while $-2$ is a lower bound for
the case in favour to the Hamiltonian, and comes out from
\eqref{DeltaHonA}.

Notice that the sets $A\in \mathcal{A}_k$ such that $A \cap
[-\ell,\ell)^d = \emptyset$ do not appear in \eqref{disuH} because,
in this case, $\mathcal{H}_{\ell ,
\mathcal{J}}(\sigma)=\mathcal{H}_{ \ell ,
\mathcal{J}}(\sigma^{(A)})$.

We check \eqref{disuH} by recursion.

Now, define
$$
\bar s =\inf\{s \in \mathcal{S}^-_{A,t} \cup \mathcal{S}^0_{A,t}
\cup \mathcal{S}^+_{A,t} : A \in \mathcal{A}_k \text{ and } A\cap
[-\ell, \ell)^d\neq \emptyset\}.
$$
By the finiteness of the set of all the $A\in \mathcal{A}_k$
intersecting $[-\ell, \ell)^d$, then there exists a unique (a.s.)
$A_1 \in \mathcal{A}_k$ with $A_1 \cap [-\ell, \ell)^d\neq
\emptyset$ such that
$$
\bar s \in \mathcal{S}^-_{A_1,t} \cup \mathcal{S}^0_{A_1,t} \cup
\mathcal{S}^+_{A_1,t}.
$$
Then
$$
\mathcal{H}_{ \ell , \mathcal{J}}(\sigma^{(A_1)} (0)) -
\mathcal{H}_{\ell  , \mathcal{J}}(\sigma (0))  $$
\begin{equation*}
\label{disuH-1} \leq kd_G \cdot \left[ \mathbf{1}_{\{A_1\in
\mathcal{A}_k : A_1 \cap [-\ell,\ell)^d  \neq \emptyset \} } \cdot
\mathbf{1}_{    \{  \bar s    \in
\mathcal{S}^-_{A_1,t}\}}+\mathbf{1}_{\{A_1\in \mathcal{Z}_\ell \}}
\right]-2\cdot \mathbf{1}_{\{A_1\in \mathcal{Z}_\ell^+\}} \cdot
\mathbf{1}_{\{ \bar s   \in \mathcal{S}^+_{A_1,t}\}}.
\end{equation*}
Thus, by induction  one obtains  \eqref{disuH}.

By Lemma \ref{Lemrichiesto} one has
\begin{equation}
\label{disuH-4} \mathbb{E}[\mathcal{D}_{V \cap  [-\ell,\ell)^d   ,
\mathcal{J}}(\sigma (t ))] \in [-d_G/2,d_G/2].
\end{equation}
By applying the expected value operator to the terms in inequality
\eqref{disuH} and by \eqref{disuH-4}, one has
$$
\mathbb{E}[\mathcal{H}_{ \ell  , \mathcal{J}}(\sigma (t ))] \leq
\frac{1}{2} d_G |V_{Cell}| (2\ell)^d
$$
\begin{equation}
\label{disuH-3}+ \mathbb{E}\left[\sum_{A\in \mathcal{A}_k : A \cap
[-\ell,\ell)^d \neq \emptyset }
d_Gk|\mathcal{S}^-_{A,t}|\right]+\mathbb{E}\left[\sum_{A\in
\mathcal{Z}_\ell } d_Gk|\mathcal{T}_{A}\cap [0,t]|\right]-2\cdot
\mathbb{E}\left[ \sum_{A\in \mathcal{Z}_\ell^+}
|\mathcal{S}^+_{A,t}|\right] ,
\end{equation}
since $\mathcal{S}^-_{A,t} \cup \mathcal{S}^0_{A,t} \cup
\mathcal{S}^+_{A,t} \subset \mathcal{T}_A \cap [0,t]$.

Now, we bound the single terms of inequality \eqref{disuH-3}.

By \eqref{val1} in Lemma \ref{serve}, one has
\begin{equation}
\label{disuH-3aa} \mathbb{E}\left[\sum_{A\in \mathcal{A}_k : A \cap
[-\ell,\ell)^d \neq \emptyset } d_Gk|\mathcal{S}^-_{A,t}|\right]\leq
\mathbb{E}\left[\sum_{A\in \mathcal{A}_k : A \cap [-\ell,\ell)^d
\neq \emptyset } d_Gk|\mathcal{S}^-_{A,\infty}|\right] \leq C_{1,k}
\ell^d ,
\end{equation}
where the constant $C_{1,k}$ can be chosen in such a way that it
does not depend on $\ell $. In fact, by translation invariance, the
expected value over  a cell is equal to the expected value on any
other cell.

By the bound on the cardinality of $ \mathcal{Z}_\ell $ we obtain
\begin{equation}
\label{disuH-3ab} \mathbb{E}\left[\sum_{A\in \mathcal{Z}_\ell }
d_Gk|\mathcal{T}_{A}\cap [0,t]|\right]  \leq d_G k
    | \mathcal{Z}_\ell |   \mathbb{E}\left [
|\mathcal{T}_{A}\cap [0,t]|  \right ]= d_G k t
| \mathcal{Z}_\ell | \leq d_G k t M_{k,G}\ell^{d-1}.
\end{equation}
Now,  by contradiction, suppose that there exists $\bar A \in
\mathcal{A}_k$ such that 
$ \mathbb{E} [ |S^+_{\bar A, \infty}| ] = \infty$. For $\ell$ large
enough, one has
\begin{equation}
\label{disuH-3ac} 2\cdot \mathbb{E}\left[ \sum_{A\in
\mathcal{Z}_\ell^+} |\mathcal{S}^+_{A,t}|\right] \geq \mathbb{E} [
|S^+_{\bar A, t}| ] \ell^d.
\end{equation}
By \eqref{disuH-3}-\eqref{disuH-3ac} and for $\ell$ large enough,
one gets
\begin{equation}
\label{disuH-5}\mathbb{E}[\mathcal{H}_{ \ell  , \mathcal{J}}(\sigma
(t ))] \leq  \frac{1}{2} d_G |V_{Cell}| (2\ell)^d  +C_{1,k} \ell^d +
d_G k t M_{k,G}\ell^{d-1} - \mathbb{E} [ |S^+_{\bar A, t}| ] \ell^d.
\end{equation}

\medskip

The monotone convergence theorem leads to $ \lim_{t \to \infty
}\mathbb{E} [ |S^+_{\bar A, t}| ] = \infty $. Let us take $t =
\ell^{1/2} $. Therefore
$$
\mathbb{E}[ D_{\Lambda, \mathcal{J}} (\sigma(\ell^{1/2} ))] \leq
\frac{1}{2} d_G + \frac{C_{1,k}}{ 2^d |V_{Cell}| } + \frac{d_G k
M_{k,G}}{ 2^d |V_{Cell}| } \ell^{-1/2}  -  \frac{1}{ 2^d |V_{Cell}|
} \mathbb{E} [ |S^+_{\bar A, \ell^{1/2}} |].
$$
If $ \mathbb{E} [ |S^+_{\bar A, \infty} |] = \infty$ the r.h.s. of
the previous inequality should become smaller than $-\frac{1}{2}
d_G$,  for  $\ell $ large. This is in contradiction with Lemma
\ref{Lemrichiesto}, therefore $ \mathbb{E} [ |S^+_{\bar A, \infty}
|] < \infty$.
%
%
%
\end{proof}
We notice  that only in the case of a temperature profile $T$ that
does not satisfy Corollary \ref{coro} one can hope to obtain
$\rho_{\mathcal{F}}>0 $.
In the following we work under the  assumption that the temperature
profile $T$ is fast decreasing to zero and we give, jointly to it,
some sufficient conditions to obtain $ \rho_{\mathcal{F}} >0$.

Next two theorems provide conditions for $\rho_{\mathcal{F}} >0$.
Analogous results are in \cite{GNS, NNS00} for the case $k =1$.


\begin{theorem} \label{culta}
Let us consider a $ (k,  \alpha , \gamma; T)$-model on a $d$-graph
$G= (V, E)$ with $k \in \mathbb{N}$, $\alpha \in
    [0,1]$, $\gamma \in [0,1]$ and $T $ fast decreasing to zero.
 Suppose that
there exists  $v \in V$ such that, for any $A \in \mathcal{A}_k$
with
   $A \ni v $, the cardinality of the set $ \Gamma_A=\{ \{u,w\}\in E : u \in \partial A ,
   w \in \partial^{ext} A \}$   is
   odd.
    Then $\rho_{\mathcal{F}}>0 $. Specifically, $\lim_{t \to \infty} \sigma_v (t) $ exists a.s..
\end{theorem}

 \begin{proof}
 Let us consider $A \in \mathcal{A}_k$ with $A \ni v $.
 As in \cite{GNS,NNS00}, we notice that if $t \in \mathcal{T}_A$ corresponds to
 a flip at $A $, then $t \in \mathcal{S}^-_{A, \infty} \cup \mathcal{S}^+_{A, \infty}$.
Formula \eqref{val1} in Lemma \ref{serve} implies that
$|\mathcal{S}_{A,t}^-|<\infty$ and Lemma \ref{Lemnonrichiesto}
states that $|\mathcal{S}_{A,t}^+|<\infty$, almost surely. Thus, we
have $|\mathcal{N}_{v}|<\infty$ almost surely. By ergodicity,
$\rho_{\mathcal{F}}>0 $.
%
\end{proof}

\begin{theorem} \label{cultb}

Let us consider a $ (k,  \alpha , \gamma; T)$-model on a $d$-graph
$G= (V, E)$ with $k \in \mathbb{N}$, $\alpha \in
    (0,1]$, $\gamma \in (0,1]$ and $T $ fast decreasing to zero.

Assume that there exists a finite set $B \subset V   $ such that
   $|B| > k $ and for any $A  \in \mathcal{A}_k$ with $A \cap B \neq \emptyset $  one has
   \begin{equation*}
     |\{ \{u,v\}\in E : u \in A \cap B , v \in A^c \cap B \}| >
      \end{equation*}
      \begin{equation}\label{grande2}
      |\{ \{u,v\}\in E : u \in A , v \in A^c \cap B^c \}|
      +|\{ \{u,v\}\in E : u \in A \cap B^c , v \in A^c \cap B \}|  .
   \end{equation}

 Then $\rho_{\mathcal{F}}>0 $.
In particular, $\mathbb{P}( \bigcap_{u \in B} \{  \lim_{t \to
\infty} \sigma_u(t) =+1 \})>0$.

\end{theorem}
 \begin{proof}

Let us consider $B$ as in the statement of the theorem. Since
$\alpha
>0$ and $E(B)$ is finite, then  with positive
probability one has that $J_e=+1$ for any $e \in E(B)$. In fact, the
independence of the interaction leads to $\mathbb{P}(\bigcap_{e \in
E(B)} \{J_e=+1\})=\alpha^{|E(B)|} >0$. Analogously, it is
$\mathbb{P}(\bigcap_{v \in B} \{\sigma_v(0)=+1\})=\gamma^{|B|}  >0
$.

Notice that $\{ \{u,v\}\in E : u \in A \cap B , v \in A^c \cap B
\}$, $\{ \{u,v\}\in E : u \in A , v \in A^c \cap B^c \}$ and $\{
\{u,v\}\in E : u \in A \cap B^c , v \in A^c \cap B \}$ is a
partition of $\{\{u , v\}\in E: u\in A, v  \in A^c\}$.

Hence, inequality (\ref{grande2}) and conditions $\sigma_v =+1$ for
each $v\in B$ and $J_e=+1$ for each $e \in E(B)$ imply that for each
$A \in \mathcal{A}_k$ such that $A \cap B \neq \emptyset$ one has
$$
\Delta_A \mathcal{H}_{\mathcal{J}}(\sigma)  = 2 \sum_{\{u , v\}\in
E: u\in A, v  \in A^c }
 J_{\{u,v\} }\sigma_u \sigma_v \geq 2|\{ \{u,v\}\in E : u \in A \cap B , v \in A^c \cap B \}|
 $$
 $$
    -2 \big [  |\{ \{u,v\}\in E : u \in A , v \in A^c \cap B^c \}|
      +|\{ \{u,v\}\in E : u \in A \cap B^c , v \in A^c \cap B \}|
      \big ] \geq 2.
 $$

By formula \eqref{fa} in Lemma \ref{serve}, there is a positive
probability that  all the spins maintain the initial value $+1$ in
the region $B$. By ergodicity, we obtain $\rho_{\mathcal{F}}
>0$.
\end{proof}

\begin{rem}
\label{remrichiamo} Let us consider the case $k=1$ (see
\cite{NNS00}). The hypothesis of Theorem \ref{culta} can be
rephrased by stating that there exists $v \in V$ such that $d_v$ is
odd. Furthermore, condition in Theorem \ref{cultb} becomes
 \begin{equation*}
     |\{ \{u,v\}\in E : u \in A \cap B , v \in A^c \cap B \}| >
      |\{ \{u,v\}\in E : u \in A , v \in A^c \cap B^c \}| .
   \end{equation*}
\end{rem}
\begin{example}
\label{example1}
 Let us consider  $k=1$ and the hexagonal planar
lattice. In this situation $\rho_{\mathcal{F}}=1$, for any $\alpha,
\gamma \in [0,1]$ and $T$ fast decreasing to zero (see Theorem
\ref{culta} and Remark \ref{remrichiamo}). However, for $k\geq 2$,
it is easy to check that the hypothesis of Theorem \ref{culta} is
not verified for such a graph. In this case, Theorem \ref{hexag}
guarantees that $\rho_{\mathcal{F}}<1$ when $\alpha=1$ and
$\gamma=1/2$.
\newline
We present here a $2$-graph on which Theorem \ref{culta} can be
applied when $k =2 $ (see Figure \ref{fig4}). All the intersections
among line segments represent sites of the graph. The site denoted
with the black bullet is $v$. Each connected set $A \ni v$ of
cardinality 2 is such that $|\Gamma_A|=5$. The case $A =\{v\}$ is
such that $|\Gamma_A|=3$.

For what concerns Theorem \ref{cultb}, we address the reader to
Figure \ref{fig1}. In this case we consider $k=1$. Take $B$ as the
set of the black bullets. $A$ is a set containing only a vertex of
$B$. It is clear that Theorem \ref{cultb} can be applied.
\end{example}

\begin{figure}
  \includegraphics[scale=0.5]{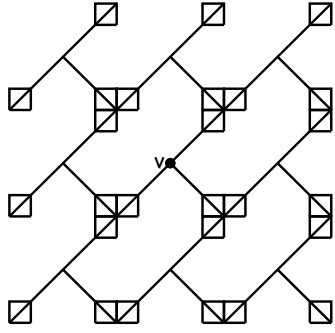}
  \caption{\small{Example of a $2$-graph on which Theorem
\ref{culta} can be applied when $k =2 $.}}
  \label{fig4}
\end{figure}
In the following results we adapt the definition of \emph{e-absent}
set in \cite{CDN02, GNS}.
\begin{definition}\label{defeabsent}
Let us consider a $d$-graph $G=(V, E)$, interactions $\mathcal{ J
}\in \{-1,+1\}^E $. Take the $(k, \langle \mathcal{J} \rangle,
\langle \sigma \rangle; T)$-model on $G$, with generic initial
configuration $ \sigma \in \{-1,+1\}^V  $, and  a
 finite $C \subset V $.  We say that $\hat \sigma_C
 =(\hat \sigma_x : x \in C)$ is k\emph{-absent on} $\mathcal{J}$ if 
there exists a
sequence of $m$ elements
$$
(A^{(j)} \in \mathcal{A}_k : \, A^{(j)} \subset C , \text { for }j
=1, \ldots , m ),
$$
 such that
$\Delta_{A^{(m)}} \mathcal{H}_{\mathcal{J}}( \sigma^{(m-1)} ) <0$,
while  $\Delta_{A^{(j)}} \mathcal{H}_{\mathcal{J}}( \sigma^{(j-1)}
)=0 $, for any $j=1, \dots, m-1$, where $ \sigma^{(0)}  $ coincides
with $\hat \sigma_C $ on $C$. 
Recursively,  $ \sigma^{(j)} =(\sigma^{(j-1)} )^{( A^{(j)} )} $, for
 $j=1, \dots, m-1$.

\end{definition}

\begin{rem} \label{importanti}
\begin{itemize}
\item[$(i)$] For any given $ \mathcal{J} $, if $\hat \sigma_C $ is
$k$-absent on $\mathcal{J}$, then it is also $k'$-absent on
$\mathcal{J}$, for each $k' > k$, being $ \mathcal{A}_k \subset
\mathcal{A}_{k'}$.

\item[$(ii)$]  Let us consider a finite set $C\subset V$. We
notice that $\hat \sigma_C$ is $k$-absent on $\mathcal{J}$ only on
the basis of the restriction of $\mathcal{J} $ to the finite set of
edges $\hat{E}(C)=\{ \{x, y \} \in E : \{x, y \} \cap C \neq
\emptyset \}$.

\item[$(iii)$] Let us consider
$ \hat \sigma_C $ $k$-absent on $\mathcal{J}$ and a
configuration $\bar \sigma \in \{-1,+1\}^V$.

If $ \bar \sigma_C $ is such that there exists a finite sequence of
$\bar m$ flips $A^{(j)} \subset C,\,A^{(j)} \in \mathcal{A}_k$, with
$\Delta_{A^{(j)}} \mathcal{H}_{\mathcal{J}}( \bar \sigma^{(j-1)} )
\leq 0$, where $ \bar \sigma^{(0)}  $ coincides with $\bar \sigma_C
$ on $C$, $ \bar \sigma^{(j)} =(\bar \sigma^{(j-1)} )^{( A^{(j)} )}
$, for each $j=1, \dots, \bar m$, and $\bar \sigma^{( \bar m)}_C =
\hat \sigma_C $, then $ \bar \sigma_C $ is $k$-absent on
$\mathcal{J}$ as well.
\end{itemize}
\end{rem}
\medskip


\begin{theorem}\label{teoabsent}
Let us consider a $(k, \alpha, \gamma; T)$-model  on a $d$-graph
$G=(V,E)$ with $k \in \mathbb{N}$, $\alpha \in (0,1)$, $\gamma \in
[0,1]$ and temperature profile $T$ fast decreasing to zero.
Moreover, let us take a finite subset $C \subset V$, interactions
$\hat{\mathcal{J}}= (\hat{J}_e:e \in E) \in \{-1,+1\}^E$ and a
configuration $\hat \sigma_C=(\hat \sigma_v:v \in C) \in
\{-1,+1\}^C$ $k$-absent on $\hat {\mathcal{J}}$.  Then the set
\begin{equation}
\label{Asigma} A(\hat \sigma_C)=\left\{ \sigma \in \{-1,+1\}^V:
\sigma_{v}= \hat \sigma_v,\,\forall\, v \in C \right\}
\end{equation}
recurs with null probability by using the probability measure
\begin{equation}\label{Pcondizionata}
    P ( \cdot ) = \mathbb{P} ( \cdot | \bigcap_{e \in \hat{E}(C)} \{ J_e
= \hat{J}_e\} ) ,
\end{equation}
where $\hat{E}(C)=\{ \{x, y \} \in E : \{x, y \} \cap C \neq
\emptyset \}$.
\end{theorem}
\begin{proof}
We consider the Harris' graphical model and denote it by $Z = (Z_t:
t \geq 0 )$. 
By hypothesis, $\mathcal{J}$ of the model coincides with
$\hat{\mathcal{J}}$ over the set $\hat{E}(C)$.

By contradiction, assume that $A(\hat \sigma_C)$ defined in
\eqref{Asigma} recurs with positive probability.

Define the random set
$$
\mathcal{K}=\bigcup_{A \in \mathcal{A}_k: A \cap C \neq \emptyset}
\mathcal{T}_A \cap [0,1].
$$
Therefore $\mathcal{K} $ is a function of the  process $Z$. Define
the set
$$
B=\Big\{(Z_t:t \in [0 , 1] ) :  \mathcal{K}= \{t_1, \dots,
t_m\}\text{ with } 0 <  t_1< \dots< t_m  <  1 \text{ and }
$$
$$
t_1 \in \mathcal{S}^0_{A^{(1)},1}, \dots, t_{m-1} \in
\mathcal{S}^0_{A^{(m-1)},1}, t_m \in \mathcal{S}^+_{A^{(m)},1}%
\Big\},
$$
where  $m $ and  the $A^{(j)}$'s are given as in Definition \ref{defeabsent}.

We notice that we are under the conditions of Lemma \ref{nuovoabs}.
In fact,
\begin{equation*}\label{estinf4}
\inf\limits_{z \in A   (\hat \sigma_C)   } \inf\limits_{t \geq 0} P(
\{ Z_{t+s} : s \in [0,1]\} \in B|Z_t=z) \geq \varsigma > 0,
\end{equation*}
because the probability that the  $U$'s are smaller than $1/2$ and
that the only finite sequence of arrivals in $[0,1]$ of the Poisson
processes $\mathcal{P}_A$, with $ A \in \mathcal{A}_k$ and $A \cap C
\neq \emptyset $, is ordered as in $\mathcal{K}$, is larger than
zero (see the proof of Theorem \ref{unione} for a similar argument). 

Since $A( \hat{\sigma}_C)$ recurs with positive probability,  then
Lemma \ref{nuovoabs} gives  that $B $ recurs  with positive
probability. In particular, one also has that $P(
|S^+_{A^{(m)},\infty} |= \infty )>0$ but $ \mathbb{P} ( \bigcap_{e
\in \hat{E}(C)} \{ J_e = \hat{J}_e\} )>0 $. Therefore $\mathbb{P}(
|S^+_{A^{(m)},\infty} |= \infty )>0$,   and this contradicts Lemma
\ref{Lemnonrichiesto}.
\end{proof}

\begin{rem} \label{remabsent}
We can extend Theorem \ref{teoabsent} also to $\alpha = 0$ (resp.
$\alpha = 1$) by selecting $\hat{ \mathcal{J}} \equiv -1$ (resp.
$\hat{ \mathcal{J}} \equiv +1$).
\end{rem}

\begin{theorem}\label{teorhoF}

Let us consider a $(k, \alpha, \gamma; T)$-model  on a $d$-graph
$G=(V,E)$ with $k \in \mathbb{N}$, $\alpha \in (0,1)$, $\gamma \in
[0,1]$ and temperature profile $T$ fast decreasing to zero. If there
exist interactions $\hat{\mathcal{J}} \in \{-1,+1\}^E$, a finite set
$C \subset V $ and $u,v \in C$ with $\nu_G( u,v) \geq k$ such that
 any  $\sigma_C \in \{ -1,+1\}^C$ with $\sigma_u =\sigma_v $
 (resp. $ \sigma_u =-\sigma_v$) is
$k$-absent on $\hat{\mathcal{J}}$. 
Then
$$
P \left (\lim_{t\to \infty} \sigma_u (t )=-\lim_{t\to \infty}
\sigma_v (t ) \right ) =1 , \text{ (resp. } P \left (\lim_{t\to
\infty} \sigma_u (t )=\lim_{t\to \infty} \sigma_v (t ) \right ) =1
\text{)},
$$
where $P $ is the conditioned probability as in
\eqref{Pcondizionata}, and $\rho_{\mathcal{F}}
>0$.
\end{theorem}
\begin{proof}\label{si}
We only prove the case when any $\sigma_C \in \{ -1,+1\}^C$, with
$\sigma_u =\sigma_v $,
  is $k$-absent on $\hat{\mathcal{J}}$. The other case is analogous.

The interactions $\mathcal{J}$ coincide with $\hat{\mathcal{J}}$ on
$\hat{E}(C)$ with positive probability, since $\alpha \in (0,1)$ and
the set $\hat{E}(C)$ is finite.

Let us suppose  that $\mathcal{J}$ coincides with
$\hat{\mathcal{J}}$ on $\hat{E}(C)$. The vertices $u$ and $v $, as
defined in the theorem, flip simultaneously with probability zero
because $\nu_G( u,v) \geq k$. In fact, for any $A \in
\mathcal{A}_k$, the set $\{u,v\} $ is not a subset of $A$. By
Theorem \ref{teoabsent}, we know that the set $A(u,v)=\{\sigma \in
\{-1,+1\}^V:\sigma_u=\sigma_v\}$ recurs with zero probability.
Therefore there exists $\lim_{t \to \infty}\sigma_u(t) $, $\lim_{t
\to \infty}\sigma_v(t) $ $P$-a.s., and they have opposite sign
$P$-a.s..

So, if we remove the conditioning assumption that $\mathcal{J}$
coincides with $\hat{\mathcal{J}}$ on $\hat{E}(C)$, we have that the
spins on the vertices $u,v$ fixate with positive probability.

By ergodicity we obtain that $\rho_{\mathcal{F}} >0$.
\end{proof}


%
 We present an example in which
 Theorem \ref{teorhoF}  can be applied.
\begin{example} \label{M}
We consider $k\leq 3$ and a 1-graph $G =(V, E)$, as in Figure
\ref{fig2}, whose cell has vertices (the cell will be scaled by
$\frac{1}{5}$ and translated by $-\frac{1}{10}$) given by
$$
V_{Cell} =\{ 1 ,\,  2,\,  3, \, 4 , \, 5 \}
$$
and  edges
$$
E_{Cell} =\{ \{ -1, 1 \} , \,\{ 0, 1 \} ,\, \{ 1, 2 \} ,\, \{ 1, 3\}
,\, \{ 2, 4 \} ,\,   \{ 3, 5 \} ,\, \{ 4, 6 \} ,\, \{ 5, 6\} \}.
$$
\begin{figure}
  \includegraphics[scale=0.4]{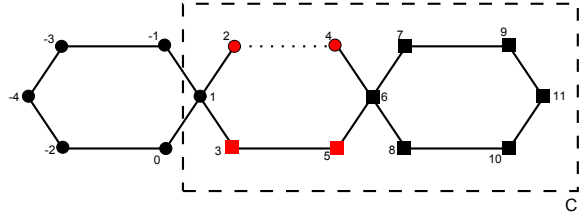}
  \caption{\small{These sites constitute a region of a graph $G=(V,E)$ with $V=\mathbb{L}_1$.
The interactions $\mathcal{J}$ are represented in the Figure. In
particular, continuous line stands for an edge $e$ such that
$J_e=+1$, while the dotted line describes $J_e=-1$. The black
bullets and squares are spins fixating definitively at a constant
configuration ($+1$ or $-1$) and, in general, bullets can be
different from squares. When bullets and squares share the same
constant configuration, then sites 2 and 4 flip infinitely often,
i.e. it does not exist $\lim_{t \to \infty} \sigma_j(t)$, with
$j=2,4$. Differently, when bullets and squares have opposite signs,
then it does not exist $\lim_{t \to \infty} \sigma_j(t)$, with
$j=3,5$. All the configurations $\sigma_C $ having $\sigma_{6} =
-\sigma_{11}$ are 1-absent on $\mathcal{J}$.
   }}
  \label{fig2}
\end{figure}

We consider as $C $ of Definition \ref{defeabsent} the vertices $\{
1,2, \ldots, 11\}$. The interactions $\mathcal{J}$, when restricted
to the elements of $C $, are taken all positive with the exception
of $J_{ \{ 2,4\}} =-1$. 
We show that
$$
A=\{\sigma \in  \{-1,+1\}^V : \sigma_6=-\sigma_{11}\}
$$
recurs with null probability. In particular, all the configurations
$\sigma_C \in \{-1,+1\}^C $ having $\sigma_{6} = -\sigma_{11}$ are
1-absent on $\mathcal{J}$.
%

Since  $\nu_G (6,11) =3$, then Theorem \ref{teorhoF} gives that
$\rho_{\mathcal{F}}>0$. By symmetry of the dynamics with respect to
global flips, i.e. $ \sigma \to -\sigma$, we can take $\sigma_6 =+1$
and $\sigma_{11}=-1$. In principle one should check $2^9$
configurations $\sigma_C $ to be 1-absent on $\mathcal{J}$.

By item $(iii)$ of Remark \ref{importanti}, there exists a finite
sequence of flips in favour of or indifferent to the Hamiltonian
leading to $\sigma_7=\sigma_8=-1 $. We assume that
$\sigma_7=\sigma_8=-1 $ for showing the 1-absence on $\mathcal{J}$
for all the configurations $\sigma_C$ having $\sigma_{6} =+1$ and
$\sigma_{11}=-1$.

We now distinguish the cases of $\sigma_1 =+1$ and $\sigma_1 =-1$.

Consider $\sigma_1 =+1$.
\begin{enumerate}
    \item If $\sigma_4 =- 1$ then, by flipping the site $6$,
we obtain $ \Delta_{\{6\}} \mathcal{H}_{\mathcal{J}} (\sigma) \leq
-4 $. Thus the $2^7$ configurations $\sigma_C \in \{-1,+1\}^C $ with
$\sigma_4=\sigma_{11}=-1 $ and $\sigma_1=\sigma_6=+1$ are $1$-absent
on the considered $\mathcal{J}$.
    \item If $\sigma_4 =+ 1$ and $\sigma_2 =+ 1$  then, by flipping
    the site $4$, we have $ \Delta_{\{4\}} \mathcal{H}_{\mathcal{J}}
    (\sigma)=0$. By flipping the site $6$,
we obtain $ \Delta_{\{6\}} \mathcal{H}_{\mathcal{J}} (\sigma) \leq
-4 $. The $2^6$ configurations $\sigma_C \in \{-1,+1\}^C $ with
$\sigma_{11}=-1 $ and $\sigma_1=\sigma_2=\sigma_4=\sigma_6=+1$ are
$1$-absent on the considered $\mathcal{J}$.
    \item If $\sigma_4 =+ 1$ and $\sigma_2 =- 1$  then, by flipping
    the site $2$, we have $ \Delta_{\{2\}} \mathcal{H}_{\mathcal{J}}
    (\sigma)=0$ and we are in case (2). The $2^6$ configurations $\sigma_C \in \{-1,+1\}^C $ with
$\sigma_2=\sigma_{11}=-1 $ and $\sigma_1=\sigma_4=\sigma_6=+1$ are
$1$-absent on the considered $\mathcal{J}$.
\end{enumerate}
The case $\sigma_1 =-1$ is analogous to the previous one by
replacing site 2 with 3 and site 4 with 5. Therefore $\lim_{t \to
\infty} \sigma_6 (t)= \lim_{t \to \infty} \sigma_{11} (t) $ and this
also implies that $\lim_{t \to \infty} \sigma_i (t) =\lim_{t \to
\infty} \sigma_6 (t) $, for $i=7,8,9,10,11$.

We notice that  all the interactions on $\hat{E}(C)$ are positive
with the exception of $ J_{\{2,4\}} =-1$ with positive probability.
In this case we can apply the arguments above and have that $\lim_{t
\to \infty} \sigma_i (t) =\lim_{t \to \infty} \sigma_1 (t) $, for
$i=-4,-3,-2,-1,0  $.

If $ \lim_{t \to \infty} \sigma_1 (t)= \lim_{t \to \infty} \sigma_6
(t) $ then the sites $  2$ and $4 $ flip infinitely many times.
Differently if $ \lim_{t \to \infty} \sigma_1 (t)= -\lim_{t \to
\infty} \sigma_6 (t) $ then the sites $ 3 $ and $ 5$ flip infinitely
many times.

This framework can be though as an illustrative example also of
Theorem \ref{unione}. In fact, the set $R$ can be taken as
coinciding with $\{-4,-3,-2,-1,0,1,6,7,8,9,10,11\}$, $S$ is the set
of the edges represented in Figure \ref{fig2}, at least one among
the four combinations of black bullets $\pm 1$ and black squares
$\pm 1$ fixates with positive probability and, accordingly, $D$ is
given by $\{2,4\}$ (bullets and squares with the same sign) or
$\{3,5\}$ (otherwise).
\end{example}

Two important remarks are in order: first, the model described in
this example is of type $\mathcal{M}$; second, this model can be
extended to $d$-dimensional graphs (see Figure \ref{fig3}).

\begin{figure}[h]
  \includegraphics[scale=0.2]{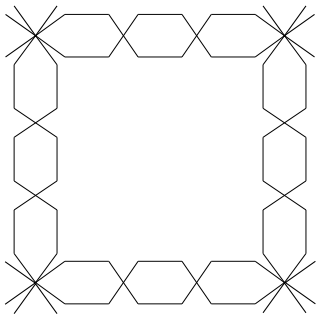}
  \caption{\small{Representation of $\Gamma_{3,3}(\mathbb{L}_2)$}.}
  \label{fig3}
\end{figure}


\begin{rem} \label{bipartito}
In \cite{NNS00}  the Ising stochastic model at zero temperature on
$\mathbb{Z}^2$ is studied.  Their model can be rewritten in our
language as $(1,1, \gamma ; T)$-model on the 2-graph $\mathbb{L}_2 =
(\mathbb{Z}^2, \mathbb{E}_2)$ with $\gamma = 1/2$ and $T \equiv 0$.
The authors show that the model is of type $\mathcal{I}$. With a
coupling argument one can prove that also for $\mathcal{J} \equiv -1
$, i.e. $\alpha =0$, the system is of type $\mathcal{I}$ (see
\cite{GNS}). In fact, consider two systems on the $d$-graph
$\mathbb{L}_d = (\mathbb{Z}^d, \mathbb{E}_d)$, $\gamma=1/2$ and
$\mathcal{J} \equiv +1 $ for one of the systems and
$\tilde{\mathcal{J}} \equiv -1 $ for the other one. We say that a
vertex $ v=(v_1, \ldots , v_d) \in \mathbb{Z}^d $ is \emph{even}
(resp. \emph{odd}) if $|\sum_{i =1}^d v_i|$ is even (resp. odd). We
take a initial configuration $\sigma $ for the system with $
\mathcal{J} $ and initial configuration $\tilde \sigma $ for the
system with $ \tilde{\mathcal{J}} $ as follows
\begin{equation*}\label{coupling}
    \tilde \sigma_v = \left \{\begin{array}{rl}
                        \sigma_v  & \text{ if } v \text{ is even;} \\
                        -\sigma_v  & \text{ if } v \text{ is odd.} \\
                      \end{array} \right .
\end{equation*}
Therefore, also the random vector $ (\tilde \sigma_v : v \in
\mathbb{Z}^d )$ is a collection of i.i.d.  Bernoulli random
variables with $\gamma =1/2$. Taking the same Poisson processes and
the same random variables $U$'s for both the systems at any time $t
$, one obtains
\begin{equation*}\label{coupling2}
    \tilde \sigma_v (t)= \left \{\begin{array}{rl}
                        \sigma_v (t) & \text{ if } v \text{ is even;} \\
                        -\sigma_v (t) & \text{ if } v \text{ is odd.} \\
                      \end{array} \right .
\end{equation*}
Therefore the two systems are of the same type ($\mathcal{I}$,
$\mathcal{F}$ or $\mathcal{M}$). The same construction works for any
bipartite graph.
\end{rem}
Example \ref{M} suggests a general result on a particular class of
graphs. We firstly define such a class, and then present a simple
and meaningful theorem. The proof of such a result can be viewed as
a generalization of the arguments carried out in Example \ref{M}.
\begin{definition}
\label{def:Gamma} For $\ell \geq 2 $ and $m \geq 1$, let us consider
a $d$-graph $G=(V,E)$, with $d \in \mathbb{N}$. We denote by
$\Gamma_{\ell,m}(G)=(V_\Gamma,E_\Gamma)$ the $d$-graph such that
each edge of $G$ is replaced with $m$ identical cycles whose number
of vertices is $2\ell$ with the following property:
\begin{itemize}
\item[$(i)$] two adjacent cycles have only one vertex in common. We
denote such vertices as \emph{common vertices}, and collect them in
the set $\mathcal{V}_I$.
\item[$(ii)$]  $\min_{  u,v \in \mathcal{V}_I}  \nu_{\Gamma_{\ell,m}(G)} (u,v)  =\ell $.
\end{itemize}
The vertices $V \subset V_\Gamma$ will be called \emph{original
vertices}.
\end{definition}
Notice that $ V\subset \mathcal{V}_I$. In Figure \ref{fig3} we have
the representation of $\Gamma_{3,3}(\mathbb{L}_2)$.
\begin{theorem}
\label{thm:ultimo} Let $G=(V,E)$ be a $d$-graph, with $d \in
\mathbb{N}$. Consider a $(k,\alpha, \gamma; T)$-model on
$\Gamma_{\ell,m}(G)$ with $\alpha \in (0,1)$, $\gamma \in [0,1]$ and
temperature profile $T$ fast decreasing to zero. If $\ell \geq k
\wedge 2 $ and $m \geq 3     $, the model is of type $\mathcal{M}$.
\end{theorem}
\begin{proof}
We briefly denote $ \nu_{\Gamma_{\ell, m }(G) }$ by $\nu$. Consider
two original vertices $u,v \in V$ such that $\nu(u,v)= \ell m$. The
finite subgraph composed by the $m$ cycles between $u$ and $v$ will
be denoted by $\Gamma^{u,v}=(V_\Gamma^{u,v},E_\Gamma^{u,v})$. The
cycles composing $\Gamma^{u,v}$ are subgraphs as well. They will be
denoted by $Y_1=(V_1,E_1), \dots, Y_m=(V_m,E_m)$, so that if one
takes $h, \ell=1, \dots, m$ such that $h<\ell$, then for each $x \in
V_h$ and $y \in V_\ell $ it results $\nu(x,u)<\nu(y,u)$.  We
consider the first three cycles $Y_1$, $Y_2$ and $Y_3$ and we label
the vertices of these cycles in the following way: $v_1=u$, $v_2 \in
V_1 \cap V_2$, $ v_3 \in V_2 \cap V_3$ and $ v_4 $ is different from
$v_3$ and belong to $\mathcal{V}_I \cap V_3$. The other vertices are
denoted by $ v_{a,b,s} $, with $ a= 1,2,3$, $b=1,2$, $s = 1, \ldots
, \ell -1 $. Specifically, the index $a$ means that $ v_{a,b,s} $
belongs to $V_a$; the index $b$ identifies one of the two paths of
length $\ell$ connecting the common vertices of the cycle $Y_a$; the
index $s$ is such that $\nu(v_{a,b,s},u)<\nu(v_{a,b,s+1},u)$, for
each $s=1, \dots, \ell-1$.

Let us consider the following condition
\begin{itemize}
\item[ (H1)]
 $ J_{ \{v_2, v_{ 2,1,1 }  \}  } =-1$;  $J_e = +1$,
for each $e \in (E_1 \cup E_2 \cup E_3) \setminus \{v_2, v_{ 2,1,1 }  \} $.
\end{itemize}
 With positive probability, there exist interactions $\mathcal{J}$
that satisfies the previous condition (H1). Define $C=V_1\cup V_2$.

We prove that the configurations belonging to
\begin{equation}\label{koko}
 \{ ( \sigma_v  \in \{ -1,+1\} : v \in C ) : \sigma_{v_1} \neq \sigma_{v_2}  \}  \subset
 \{-1,+1\}^C
\end{equation}
are 1-absent on the interactions $\mathcal{J}$ that satisfy
condition (H1). Without loss of generality in \eqref{koko}  one can
take $\sigma_{v_1} =+1 $ and  $\sigma_{v_2} =-1 $. Now we consider
two cases
\begin{itemize}
\item[(i)]  $\sigma_{v_1} = \sigma_{v_3} =+1 $ and  $\sigma_{v_2} = -1 $;
\item[(ii)]  $\sigma_{v_1}  =+1 $ and  $\sigma_{v_2} = \sigma_{v_3}= -1 $.
\end{itemize}
In  case (i) let us consider the following elements
\begin{equation*} \label{cappello}
\hat A^{(i)}  =\left \{
 \begin{array}{ll} \{ v_{1,1,i} \}  ,  &\text{ for } i = 1, \ldots , \ell-1; \\
 \{ v_{1,2,i- \ell +1} \}  ,  &\text{ for } i = \ell, \ldots , 2\ell-2; \\
  \{ v_{2,2,   3\ell -2- i } \}  ,  &\text{ for } i = 2\ell -1, \ldots , 3\ell-3; \\
   \{ v_{2 } \}  ,  &\text{ for } i = 3\ell-2. \\
 \end{array}
    \right .
\end{equation*}
For a given $\sigma_C =( \sigma_v : v \in C)$ satisfying item (i) we
eliminate  from  the sequence $   (     \hat A^{(i)}   : i=1, \ldots
, 3\ell - 2  ) $ the elements of the form $\{ v \} $ with $\sigma_v
=+1 $ to obtain the new sequence $(  A^{(i)}   : i=1, \ldots , m  )
$ (thus $m \leq 3\ell -2$). Using these elements and notation  in
Definition   \ref{defeabsent}, one can check that all the
$\Delta_{A^{(i)} }  \mathcal{H}_{\mathcal{J}} (\sigma^{(i-1)} ) \leq
0 $ for $ i =1, \ldots , m-1$, and $\Delta_{A^{(m)} }
\mathcal{H}_{\mathcal{J}}     (\sigma^{(m-1)} ) < 0 $. Notice that
$m \geq 1$, and $A^{(m)}=\{v_2\}$. Therefore, by using  item (iii)
of Remark \ref{importanti}, one has that all the configurations in
which item (i) holds true are 1-absent on $\mathcal{J}$.

In  case  (ii), we define
\begin{equation*} \label{cappello2}
\tilde A^{(i)}  =\left \{
 \begin{array}{ll} \{ v_{1,1,i} \}  ,  &\text{ for } i = 1, \ldots , \ell-1; \\
 \{ v_{1,2,i- \ell +1} \}  ,  &\text{ for } i = \ell, \ldots , 2\ell-2; \\
  \{ v_{2,1,   3\ell -2- i } \}  ,  &\text{ for } i = 2\ell -1, \ldots , 3\ell-3; \\
   \{ v_{2 } \}  ,  &\text{ for } i = 3\ell-2. \\
 \end{array}
    \right .
\end{equation*}
For $\sigma_C =( \sigma_v : v \in C)$ which satisfies item (ii) we
eliminate  from  the sequence $   (     \hat A^{(i)}   : i=1, \ldots
, 3\ell - 2  ) $ the elements of the form $\{ v_{1,b,s} \} $ with
$\sigma_{ v_{1,b,s} } =+1 $ and $\{ v_{2,1,s} \} $ with $\sigma_{
v_{2,1,s} } =-1 $. In so doing,  we have the new sequence $( A^{(i)}
: i=1, \ldots , m  ) $. By following the argument developed in case
(i),  all the configurations in which item (ii) is satisfied are
1-absent on $\mathcal{J}$.

Since $\nu(v_1,v_2)=\ell \geq k$, then Theorem \ref{teorhoF}
guarantees that, for large $t $, $\sigma_{v_1} (t ) $ and
$\sigma_{v_2} (t ) $ fixate and are equal. Therefore,
$\rho_{\mathcal{F}} >0$.

With an analogous argument one can prove the same result for $v_3$
and $v_4$, i.e. there exist the limits $ \lim_{t \to \infty }
\sigma_{v_3} (t) $ and $ \lim_{t \to \infty } \sigma_{v_4} (t) $,
and they coincide.

As in Theorem \ref{unione} we take $R = \{v_2, v_3\} $. In fact,
condition \eqref{thm4new1} is verified, i.e.: with positive
probability, at least one of the following four alternatives is
verified: $\lim_{t \to \infty }\sigma_{v_2}(t)=\pm 1$ and $\lim_{t
\to \infty }\sigma_{v_3}(t)=\pm 1$. We analyze separately the four
cases.

If $\lim_{t \to \infty }\sigma_{v_2}(t)=+ 1$ and $\lim_{t \to \infty
}\sigma_{v_3}(t)=+ 1$ with positive probability, then we select
$R'=\{v_{2,1,s}:s =1,\dots, \ell-1\}$. Let us consider
$$
\bar{s}_1=\sup\{s\in \{1, \dots, \ell-1\}:\sigma_{v_{2,1,s}}=-1\},
$$
with $\sup \emptyset=1$. We define $D=\{v_{2,1,\bar{s}_1}\}$. Notice
that $\sigma_{v_{2,1,\bar{s}_1+1}}=+1$, with the conventional
agreement that $v_{2,1,\ell}=v_3$. Then, condition (ii) of Theorem
\ref{unione} is satisfied and the model is of type $\mathcal{M}$.

An analogous argument applies to the case of $\lim_{t \to \infty
}\sigma_{v_2}(t)=\lim_{t \to \infty }\sigma_{v_3}(t)=-1$.

If $\lim_{t \to \infty }\sigma_{v_2}(t)=+ 1$ and $\lim_{t \to \infty
}\sigma_{v_3}(t)=- 1$ with positive probability, then
$R'=\{v_{2,2,s}:s =1,\dots, \ell-1\}$. In this case
$$
\bar{s}_2=\sup\{s\in \{1, \dots, \ell-1\}:\sigma_{v_{2,2,s}}=+1\},
$$
with $\sup \emptyset=1$. By taking $D=\{v_{2,2,\bar{s}_2}\}$ and
noticing that $\sigma_{v_{2,2,\bar{s}_2+1}}=-1$, with
$v_{2,2,\ell}=v_3$, we obtain condition (ii) of Theorem \ref{unione}
and the model is of type $\mathcal{M}$.

In a similar way we can treat the case of $\lim_{t \to \infty
}\sigma_{v_2}(t)=-1$ and $\lim_{t \to \infty }\sigma_{v_3}(t)=+1$.

\end{proof}



We recall the definition of cluster that will be used in the next
Theorem. Let us consider  a $d$-graph $G= (V, E)$ and a stochastic Ising model
$\sigma (\cdot ) = ( \sigma_u (t) : u \in V, t \geq 0   ) $.
\emph{The
cluster $C_v (t)$ of the site $v$ at time $t$} is the maximal
connected component of the set
$\{u \in V : \sigma_u(t)=\sigma_v(t)\} $ which contains $v$.


\begin{theorem}
\label{thm:ultimissimo} Let us consider a $(1,\alpha, 1/2; T)$-model
on $\Gamma_{\ell,m}(\mathbb{L}_2)$, with $\alpha \in \{0,1\}$, $\ell
\geq 2$, $m \geq 1$ and temperature profile $T$ fast decreasing to
zero. Then the $(1,\alpha, 1/2; T)$-model is of type $\mathcal{I}$.
\end{theorem}
\begin{proof}
First of all, we notice that $\Gamma_{\ell,m}(\mathbb{L}_2)$ is a
bipartite graph. Then, by Remark \ref{bipartito}, it is sufficient
to prove the result for $ \alpha =1$.

 The origin $O $ of the graph is placed in $(0,0) \in \mathbb{Z}^2$.
 Also in this case, for the sake of simplicity, we will write $\nu $ instead of  $ \nu_{\Gamma_{\ell, m }(\mathbb{L}_2) } $.

 We first show that if
 $$
 p=\mathbb{P}(\text{There exists }
 \lim_{t \to \infty } \sigma_O (t) =+1)= \mathbb{P}(\text{There exists }
 \lim_{t \to \infty } \sigma_O (t) =-1)=0$$
  then the model is of type $\mathcal{I}$.
 By  translation invariance of the graph, each
 vertex belonging to  $\mathbb{Z}^2 $  reaches the fixation
   with null probability.

 For a vertex $x \notin \mathbb{Z}^2 $
 let us take the two vertices $u, v \in \mathbb{Z}^2$ having Euclidean distance equal to one and
 such that $\nu (x, u), \nu (x,v) < \nu (u,v )$. The fact that $p=0 $ means that the event
 $$
 A=\{\sigma \in \{-1,+1\}^{V_\Gamma} : \sigma_u\neq \sigma_v\}.
 $$
 recurs with probability one.

 Now, define the sets
 $$
 B^+=\{(\sigma(s):s \in [0,1]):\exists \,t \in [0,1] \text{ such that }
 \sigma_x(t)=+1\}
 $$
 and
 $$
 B^-=\{(\sigma(s):s \in [0,1]):\exists \,t
 \in [0,1] \text{  such that  } \sigma_x(t)=-1\}
 $$
 One can check that
 $$
 \inf_{\sigma \in A}\inf_{t \geq 0}
 \mathbb{P}(    (\sigma(t+s): s \in [0,1])    \in    B^+|\sigma{(t)}=\sigma) >0,
 $$
 the same for $B^-$. In fact, the definition of $u,v$ and  $x$,
 along with the opposite signs of $\sigma_u$ and $\sigma_v$, allow to
 have the propagation of the sign of $u$ or $v$ over the vertex
 $x$ with positive probability.  Since $A$ recurs with probability one,
 by Lemma
 \ref{nuovoabs}, we get that also $B^+$ and $B^-$ recur with
 probability one. This means that the model is of type
 $\mathcal{I}$.

 Now the proof proceeds by contradiction. Suppose that
\begin{equation}\label{conv}
p=\mathbb{P}\left  (   \text{There exists } \lim_{t \to \infty }
\sigma_O (t) =+1 \right )
  =\mathbb{P}\left  (\text{There exists } \lim_{t \to \infty } \sigma_O (t) =-1 \right )>0.
\end{equation}

 The second equality in \eqref{conv} follows
     by $\gamma = 1/2$.
 Then, by the ergodic theorem, one should have that $\rho_{\mathcal{F}} >0$.

By the ergodic theorem, for any
 $\varepsilon >0 $ there exists an $L \in \mathbb{N}$ such that
 \begin{equation}\label{dar}
  \mathbb{P} ( \limsup_{t \to \infty}   | \{ v\in [-L, L]^2 :
  \sigma_v (t) =-1   \}  |  =0  )  <  \varepsilon .
  \end{equation}
 The previous inequality  means that the event
 $$
\hat A_L =  \{ \sigma_v=+1, \,\, \forall \, v \in [-L,L]^2\}
  $$
  recurs with probability smaller than $\varepsilon $.

Now we notice that all the finite clusters are $1$-absent on
$\mathcal{J}\equiv  +1$. In fact, given a cluster $C_v(t)$, there
exists a sequence of flips, which are indifferent for or in favour
of the Hamiltonian, such that all the spins associated to the
vertices of $C_v(t)$ change their sign (see Theorem
\ref{teoabsent}). This means that, for $v \in V_\Gamma $ and $M \in
\mathbb{N}$,
$$
\lim_{t \to \infty } \mathbb{P}(  | C_v (t)| > M ) =1.
$$
Therefore,
$$
\lim_{t \to \infty } \mathbb{P} (   C_O (t) \cap \partial   [-L,
L]^2  \neq \emptyset  ) =1.
$$
Let us define the four events associated to the four sides of
$\partial [-L, L]^2$
 $$
E_1 (t)= \{ \sigma_O(t)=+1 \text{ and } C_O (t ) \cap ( [- L, L ]
\times \{ L\} )   \neq \emptyset
       \}      ,
 $$
 $$
E_2 (t)=  \{ \sigma_O(t)=+1 \text{ and } C_O (t ) \cap ( [- L, L ]
\times \{ -L\} )   \neq \emptyset
       \}          ,
 $$
 $$
E_3 (t)= \{  \sigma_O(t)=+1 \text{ and } C_O (t ) \cap (\{ L\}\times
[- L, L ] )   \neq \emptyset
       \},
 $$
 and
  $$
 E_4 (t)= \{  \sigma_O(t)=+1 \text{ and } C_O (t ) \cap (\{ -L\}\times
[- L, L ] )   \neq \emptyset
       \}     .
  $$%
  By symmetry of the graph, the events $E_1(t), E_2(t),E_3(t), E_4(t)$ have the same probability.
  Therefore, using also that $\gamma = 1/2$, one obtain
 $$
 \liminf_{t \to \infty }   (E_i (t)) \geq \frac{1}{8}, \qquad
 \forall \, i =1, 2,3,4.
 $$
We notice that $E_i(t)$ is an increasing event, for each $i =
1,2,3,4$.

Let us define
$$
E (t )= \bigcap_{i =1}^4 E_i(t) .
 $$

By the fact that the events $E_i $'s are increasing and $k =1$ one
can use the FKG inequality to bound the probability of $E(t)$ (see
e.g. \cite{Liggett}). Therefore
 $$
 \liminf_{t \to \infty } \mathbb{P}   (E(t)) \geq
  \liminf_{t \to \infty }  \left (  \mathbb{P}   (E_N(t))
   \mathbb{P}  (E_E(t))  \mathbb{P}    (E_S(t))  \mathbb{P}  (E_O(t)) \right )   \geq
  \left (
 \frac{1}{8} \right )^4.
 $$

Consider the sequence of events  $(E (n) : n \in \mathbb{N} )$;
by Fatou's Lemma,
$$
  \mathbb{P}    ( \limsup_{n \to \infty } E(n)) \geq \limsup_{n \to \infty }
\mathbb{P}   (  E(n)) \geq      \left (
 \frac{1}{8} \right )^4      .
$$

In the frame of Lemma \ref{nuovoabs} this means that the event
$$
A_L =\{ O \text{ is connected to the four sides of
  $[-L,L]^2$ with $+1$ spins} \}
$$
 recurs with a positive probability  larger or equal to
 $\left (
 \frac{1}{8} \right )^4$.

Let us define the event
$$
B_L= \{ \exists s \in [0,1]: \sigma_u(s)=+1,\,\forall\, u \in [-L,
L]^2   \} .
$$
 One has
 $$
\inf_{\sigma\in  A_L} \inf_{t \geq 0} \mathbb{P} (       (\sigma(t+s): s \in [0,1])    \in        B _L| \sigma{(t)}
= \sigma )>0 .
 $$
 In fact, if the set of vertices with negative spins is not empty, then there exists
 at least one vertex that can change its spin from $-1$ to $+1 $ with a
flip that is
  indifferent for or in favour of the Hamiltonian. Then, recursively, all the
  spins of the vertices
 belonging to $[-L, L]^2$ can become positive at a same time.
By Lemma \ref{nuovoabs},  one obtains that  $ B_L $ recurs with
probability larger or equal than   $\left(\frac{1}{8} \right)^4$.
Therefore, for each $L \in \mathbb{N}$, the probability that there
are not vertices belonging to $[-L, L]^2 $ whose spins fixate to
$-1$ is at least
   $\left(\frac{1}{8} \right)^4$. This last assertion contradicts formula \eqref{dar} when
$\varepsilon < \left(\frac{1}{8} \right)^4$.
\end{proof}


%

\section{Conclusions}

In this paper we have presented a generalization of the Glauber
dynamics of the Ising model. The temperature is assumed to be
time-dependent and fast decreasing to zero, hence including the case
of $T \equiv 0$. Moreover, it is allowed that spins flip
simultaneously when belonging to some connected regions. The
dynamics is taken over general periodic graphs embedded in
$\mathbb{R}^d$. The obtained results  can be compared with the
standard case of zero-temperature, cubic lattice and $k=1$.

For the cubic lattice
  $\mathbb{L}_2$ the paper \cite{NNS00} says that for $\alpha =1$ or
  $\alpha =0 $ and $\gamma =\frac{1}{2}$ the model is of type
  $\mathcal{I}$. On the same graph, \cite{GNS} proves that the model is of type
  $\mathcal{M}$ when $\alpha \in (0,1)$ and $\gamma =\frac{1}{2}$
  (actually,
  their arguments hold true for $\gamma \in (0,1)$ as well). In \cite{FSS,
  Morris}  it is shown that, for $\mathbb{L}_d$ with $d\geq 2$,
  the model is of type $\mathcal{F}$ when $\alpha
  =1$ and $\gamma $ sufficiently close to one (or, by symmetry, sufficiently close to
  zero). In particular, it is known that the limit configuration is
   given by spins whose values are $+1$ (or, by symmetry, $-1$). For a better
   visualization of the results,
   see Figure \ref{fig5} (left panel).

   For what concerns the graphs $\Gamma_{\ell,m}(G)$ of Definition
   \ref{def:Gamma}, see Figure \ref{fig5} (central panel).
   In this case, when $\ell \geq k \wedge 2$ and $m \geq 3$,
   we have shown that the $(k, \alpha, \gamma;T)$-model is of type $\mathcal{M}$ for $\alpha \in (0,1)$
 and $\gamma \in [0,1]$. The particular case of
$\Gamma_{\ell,m}(\mathbb{L}_2)$ gives that the $(1, \alpha,
1/2;T)$-model is of type $\mathcal{I}$ for $\alpha=0,1$.

When $T$ is fast decreasing to zero and positive, and by considering
 $\mathbb{L}_d$ with $d \geq 2$, then it is possible to
exclude that the $(k, \alpha, \gamma;T)$-model is of type
$\mathcal{F}$ (see  Theorem \ref{infiniti}) and refer to Figure
\ref{fig5} (right panel).

We feel that Figure \ref{fig5} might also contribute to highlight
some problems left open by this paper.

To conclude, we  think that our results may represent a first move
towards the following three conjectures.
\begin{itemize}

\item[$(i)$] If $\alpha, \gamma \in (0,1)$ and $T$ is fast decreasing to
zero, then the type of the $(k, \alpha, \gamma;T)$-model over a
$d$-graph $G=(V,E)$ is identified by the value of $k \in \mathbb{N}$
and by the graph $G$.

\item[$(ii)$] Fix $\alpha, \gamma \in [0,1]$, a temperature profile fast
decreasing to zero $T$ and a $d$-graph $G$. Set $\mathcal{F}$,
$\mathcal{M}$ and $\mathcal{I}$ to 1,2,3, respectively. Define the
function $\xi_{\alpha, \gamma, T, G}:\mathbb{N} \rightarrow
\{1,2,3\}$ assigning to any $k \in \mathbb{N}$ the type of the $(k,
\alpha, \gamma;T)$-model over $G$. The function $\xi_{\alpha,
\gamma, T, G}$ is nondecreasing.

\item[$(iii)$] Let $G=\mathbb{L}_d$, for $d \geq 2$.
Consider a $(k,1, \gamma; T)$-model on $\Gamma_{\ell,m}(G)$.
 Then there exists $\varepsilon \in (0,1/2) $ and
 a temperature profile $T^\star $ fast decreasing to zero such that:
for each $\gamma \in [0,\varepsilon)\cup (\varepsilon,1]$ and for
each temperature profile $T$ such that $T(t) \leq T^\star(t)$, for
each $t \geq 0$, then the $(k,1, \gamma; T)$-model is of type
$\mathcal{F}$.
\end{itemize}

\begin{figure}[h]
  \includegraphics[scale=0.5]{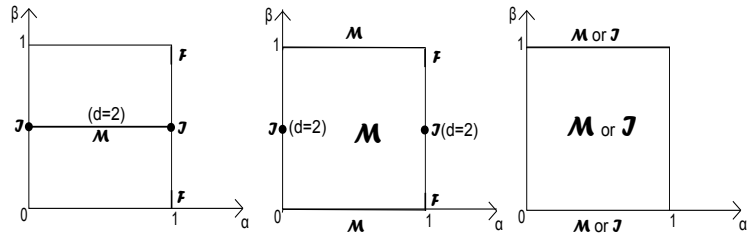}
  \caption{\small{Graphical visualization of the conclusions.
  The panel on the left describes the main results in the
  literature for the stochastic Ising models at
  zero-temperature on the graph $\mathbb{L}_d$. Notice that
  the case of $\gamma=1/2$ holds only for $d=2$. The central panel
  represents our results for the graphs $\Gamma_{\ell,m}(\mathbb{L}_d)$.
  The panel on the right depicts Theorem \ref{infiniti}. }}
  \label{fig5}
\end{figure}

\subsection*{Acknowledgement}
We thank two anonymous referees for their many insightful comments
and suggestions. We thank also Lorenzo Bertini, Emanuele Caglioti
and Mauro Piccioni for helpful discussions. All the remaining errors
are our own responsibility.

%

\end{document}